\documentclass[12pt]{amsart}

\usepackage[a4paper, centering]{geometry}
\usepackage{graphicx,bm}

\usepackage{slashed}

\usepackage[breaklinks,colorlinks]{hyperref}
\usepackage{verbatim}

\usepackage{cases}
 
\usepackage{empheq}

\usepackage{color}

\numberwithin{equation}{section}

\newtheorem{thm}{Theorem}[section]

\newtheorem{cor}[thm]{Corollary}
\newtheorem{lem}[thm]{Lemma}

\newtheorem{prop}[thm]{Proposition}

\newtheorem{defn}[thm]{Definition}
\newtheorem{rmk}[thm]{Remark}

\newcommand{\abs}[1]{\left\vert#1\right\vert}
\newcommand{\set}[1]{\left\{#1\right\}}

\newcommand{\To}{\longrightarrow}

\newcommand{\pt}{\partial}

\allowdisplaybreaks

\begin{document}

\title[Dirac-harmonic maps with the trivial index ]{Dirac-harmonic maps with the trivial index}
\makeatletter
\def\author@andify{%
  \nxandlist {\unskip ,\penalty-1 \space\ignorespaces}%
    {\unskip {} \@@and~}%
    {\unskip \penalty-2 \space \@@and~}%
}
\makeatother
\author[J\"urgen Jost]{J\"urgen Jost}
\address{Max Planck Institute for Mathematics in the Sciences, Inselstrasse 22, 04103 Leipzig, Germany}
\email{jost@mis.mpg.de}

\author[Linlin Sun]{Linlin Sun}
\address{School of Mathematics and Statistics, Wuhan University, Wuhan 430072, China}
\email{sunll@whu.edu.cn}

\author[Jingyong Zhu]{Jingyong Zhu}
\address{College of Mathematics, Sichuan University, Chengdu 610065, China}
\email{jzhu@scu.edu.cn}
\thanks{The third author wants to thank the support by the National Natural Science Foundation of China (Grant No.  12201440) and the Fundamental Research Funds for the Central Universities.}


% ----------------------------------------------------------------

\begin{abstract}
For a homotopy class $[u]$ of maps between a closed Riemannian manifold $M$ and a general manifold $N$, we want to find a  Dirac-harmonic map with the map component in the given homotopy class. Most known results require the index to be nontrivial. When the index is trivial, the few known results are all constructive and produce uncoupled solutions. In this paper, we define a new quantity. As a byproduct of proving the homotopy invariance of this new quantity, we find  a new simple proof for the fact that all  Dirac-harmonic spheres in surfaces are uncoupled. More importantly, by using the homotopy invariance of this new quantity, we prove the existence of Dirac-harmonic maps from manifolds in the trivial index case. In particular, when the domain is a closed Riemann surface, we prove the short-time existence of the $\alpha$-Dirac-harmonic map flow in the trivial index case. Together with  the density of the minimal kernel, we get an existence result for Dirac-harmonic maps from closed Riemann surfaces to K\"ahler manifolds, which extends the previous result of the first and third authors. This establishes a general existence theory for Dirac-harmonic maps in the context of trivial index.
\end{abstract}
\maketitle

% ----------------------------------------------------------------
\vspace{2em}
\section{Introduction}

Motivated by the supersymmetric nonlinear sigma model from quantum field theory, see \cite{jost2009geometry}, Dirac-harmonic maps from Riemann surfaces (with a fixed spin structure) into Riemannian manifolds were introduced in \cite{chen2006dirac}. They generalize harmonic maps and harmonic spinors. From the variational point of view, they are critical points of a conformally invariant action functional. The Euler-Lagrange equation then is an elliptic system coupling a harmonic map type equation with a Dirac type equation.

The existence of Dirac-harmonic maps from closed spin manifolds is a very difficult problem. So far, there are only a few results in this direction. Most solutions found so far are uncoupled in the sense that the map part is harmonic. The  existence result of \cite{ammann2013dirac} for uncoupled solutions depends  on the  index $I(M,u)$ being non-zero (see Definition \ref{index alpha}). But when the domain and target are both closed Riemann surfaces, the index $I(M,u)$ always vanishes.  In this case, an existence result about uncoupled Dirac-harmonic maps was proved in \cite{chen2015dirac} by the Riemann-Roch formula. Later, this result was generalized to K\"ahler manifolds in  \cite{sun2018note}.
 A general existence result for Dirac-harmonic maps from closed Riemann surfaces to compact manifolds was first established in  \cite{jost2019alpha}. This implies the existence of Dirac-harmonic maps when the index $I(M,u)$ is nontrivial.

This then naturally raises the question of  the existence of Dirac-harmonic maps in a given homotopy class $[u]$ between manifolds $M,N$ with trivial index $I(M,u)$. Of course, we should first identify conditions under which this index vanishes. When the domain $M$ is a  closed Riemann surface with  positive genus and the target $N$ is an odd-dimensional oriented manifold, there is always a spin structure on $M$ such that the index $I(M,u)$ is nontrivial. When, in contrast, $M$ is a closed Riemann surface and $N$ is an even-dimensional spin manifold, the index $
I(M,u)$ is always zero. Therefore, here we consider the case where the target manifold $N$ is a  K\"ahler spin manifold. In this case, it is necessary to use a new quantity that can  replace the index $I(M,u)$. For this purpose, we introduce a candidate that  uses the complex structure of the target manifold. More precisely, we first decompose the twisted Dirac operator as
 \begin{equation*}
 \slashed{D}^u=\slashed{D}_{1,0}^u+\slashed{D}_{0,1}^u
 \end{equation*}
according to the decomposition $(u^*TN)^{\mathbb C}=u^*T_{1,0}N\oplus u^*T_{1,0}N$. Then we just consider the kernel of one of the two operators, such as $\slashed{D}_{1,0}$. We define
\begin{equation*}
    \mathcal{I}(M, u^*T_{1,0}N):=\left[\frac12{\rm dim}_{\mathbb{C}}{\rm ker}\slashed{D}_{1,0}^{u}\right]_{\mathbb{Z}_2},
\end{equation*}
for an even dimensional spin manifold $M$ whenever the complex dimension of the kernel of $\slashed{D}_{1,0}^{u}$ is even.

In order to be useful for our purposes, this should be  homotopy invariant. Let us first look at an example. Suppose $M=\mathbb{C}{\rm P}^1$ and $N$ is a compact surface. Consider any map $u: M\to N$, the spinor bundle $\Sigma\mathbb{C}{\rm P}^1$ and the twisted bundle $\Sigma\mathbb{C}{\rm P}^1\otimes u^*T_{1,0}N$. Let $g_N$ be the genus of $N$ and $c_1\left(u^*T_{1,0}N\right)=a\gamma, a=2\deg(u)(1-g_N)$, where $\gamma$ is the tautological bundle of $\mathbb{C}{\rm P}^1$. The unique spin structure of $\mathbb{C}{\rm P}^1$ is determined by $\gamma$ since  $\Lambda^{1,0}\mathbb{C}{\rm P}^1=\gamma^2$. Then as a holomorphic bundle, we have
\begin{align*}
   \Sigma\mathbb{C}{\rm P}^1\otimes u^*T_{1,0}N=\left(\gamma\oplus\Lambda^{0,1}\mathbb{C}{\rm P}^1\otimes\gamma\right)\otimes\gamma^{a}=\gamma^{a+1}\oplus\Lambda^{0,1}\mathbb{C}{\rm P}^1\otimes\gamma^{a+1}.
\end{align*}
Since
\begin{align*}
    \dim_{\mathbb{C}}H^0\left(\mathbb{C}{\rm P}^1,\gamma^{m}\right)=\begin{cases}
    0,&m>0,\\
    1-m,&m\leq0,
    \end{cases}
\end{align*}
we conclude that
\begin{align*}
    \dim_{\mathbb{C}}\ker\slashed{D}^{u}_{1,0}=&\dim_{\mathbb{C}}H^0\left(\mathbb{C}{\rm P}^1,\gamma^{a+1}\right)+\dim_{\mathbb{C}}H^1\left(\mathbb{C}{\rm P}^1,\gamma^{a+1}\right)\\
    =&\dim_{\mathbb{C}}H^0\left(\mathbb{C}{\rm P}^1,\gamma^{a+1}\right)+\dim_{\mathbb{C}}H^0\left(\mathbb{C}{\rm P}^1,\gamma^{1-a}\right)\\
    =&\abs{a}=2|{\rm deg}(u)(g_N-1)|.
\end{align*}
Therefore, $\dim_{\mathbb{C}}\ker\slashed{D}^{u}_{1,0}$ is invariant in the homotopy class $[u]$. This implies the homotopy invariance of $\mathcal{I}(\mathbb{C}{\rm P}^1, u^*T_{1,0}N)$, which is equal to $\left[|{\rm deg}(u)(g_N-1)|\right]_{\mathbb{Z}_2}$. Moreover, the dimension of the kernel of the Dirac operator is a constant in a given homotopy class.  
Then the following well-known fact follows from the
first variational formula.

\begin{prop}[\cite{yang2009structure}]
There is no coupled Dirac-harmonic map from the $2$-sphere into a compact Riemann surface.
\end{prop}

In general, we can give two different sufficient conditions to guarantee the homotopy invariance of $\mathcal{I}$.

\begin{thm}\label{inv}
Suppose that $M$ is an even dimensional spin Riemannian  manifold and $(N,i)$ is a K\"ahler manifold. If one of the following holds:
\begin{itemize}
\item[(1)]the complex spinor bundle $(\Sigma M, i_1)$ over $M$ admits a commuting real structure $j$, i.e. a  real structure 
($
j^2=id_{\Sigma M}, \ ji_1=-i_1j$)
commutes with
Clifford multiplication
 and $N$ is hyperK\"ahler;
\item[(2)] the complex spinor bundle $\Sigma M$ over $M$ admits a commuting quaternionic structure $j_1$ and there exists a parallel real structure $j_2$ on $T_{1,0}N$, i.e.
\begin{equation*}
j_2^2=id_{T_{1,0}N}, \ j_2i=-ij_2 \ \text{and} \ \nabla{j_2}=0.
\end{equation*}
\end{itemize}
Then all the eigenspaces of $\slashed{D}^u_{1,0}$ are quaternionic vector spaces for any map $u:M\to N$ and $\mathcal{I}(M, u^*T_{1,0}N)$ is invariant in the homotopy class $[u]$.

Moreover, if $\mathcal{I}(M, u^*T_{1,0}N)\neq0$, then there is a real vector space of  real dimension $\ge 4$ such that all $(\tilde{u},\psi)$ are uncoupled $\alpha$-Dirac-harmonic maps as long as there is an $\alpha$-harmonic map $\tilde{u}\in[u]$ for $\alpha\geq1$.
\end{thm}

Here  $\alpha$-Dirac-harmonic maps are the critical points of the following functional
 \begin{equation*}
 L^\alpha(u,\psi)=\frac12\int_M(1+|du|^2)^\alpha+\frac12\int_M\langle\psi,\slashed{D}^u\psi\rangle_{\Sigma M\otimes u^*TN}, \ \ \forall \alpha\geq1.
 \end{equation*}
 They are generalizations of Dirac-harmonic maps (i.e. the case of $\alpha=1$). As generalizations of harmonic maps, $\alpha$-harmonic maps are the critical points of the following functional
  \begin{equation*}
 E_\alpha(u)=\frac12\int_M(1+|du|^2)^\alpha, \ \ \forall \alpha\geq1,
 \end{equation*}
which was introduced by Sacks and Uhlenbeck in \cite{sacks1981existence}.

By the statement in  \cite[Theorem 2.2.2]{hermann2012dirac}, such a commuting real structure in Theorem \ref{inv} always exists on $M$ if $m=0,6,7 ({\rm mod} \ 8)$. In particular, when $m=0,6({\rm mod} \ 8)$, we can get the existence of uncoupled Dirac-harmonic maps.

\begin{cor}\label{direct cor}
Let $m$ be the dimension of $M$. Suppose one of the following holds:
\begin{itemize}
    \item[(a)] $m=0,6({\rm mod} \ 8)$, $N$ is a hyperK\"ahler manifold and a homotopy class $[u]$ satisfies $\mathcal{I}(M, u^*T_{1,0}N)\neq0$;
    \item[(b)] $m=2,4({\rm mod} \ 8)$, $N$ is a  K\"ahler manifold with a parallel real structure $j_2$ defined in Theorem \ref{inv} and a homotopy class $[u]$ satisfies $\mathcal{I}(M, u^*T_{1,0}N)\neq0$.
\end{itemize}
Then there is a real vector space of  real dimension $\geq 4$ such that all $(\tilde{u},\psi)$ are uncoupled $\alpha$-Dirac-harmonic maps as long as there is an $\alpha$-harmonic map $\tilde{u}\in[u]$ for $\alpha\geq 1$. 

\end{cor}

\begin{rmk}
Note that the case $m=6({\rm mod} \ 8)$ is not included in \cite{ammann2013dirac} due to the definition of index $\alpha(M,u)$.
When $m=0({\rm mod} \ 4)$, the triviality of the index $\alpha(M,u)$  implies that of the index ${\rm ind}(\slashed{D}^+)$, where $\slashed{D}^+$ comes from the decomposition of the Dirac operator according to that of the spinor bundle (see the Section 2). When $m=2k$ and $k$ is odd, the triviality of ${\rm ind}(\slashed{D}^+)$ does not imply that of ${\rm ind}(\slashed{D}^+_{1,0})$. And the second author used the non-trivial index ${\rm ind}(\slashed{D}^+_{1,0})$ to get an existence result in \cite{sun2018note}.
Our corollary is still valid even if ${\rm ind}(\slashed{D}^+_{1,0})=0$. For example, our result applies to the case when the dimensions of the kernels in those four subspaces in the decomposition \eqref{decomposition} are all equal to one, which is never considered in literatures.
\end{rmk}

When $M$ is a closed Riemann surface, we can prove the short-time existence of $\alpha$-Dirac-harmonic map flow into a  K\"ahler manifold, which generalizes the result in \cite{jost2019short}.

The rest of the paper is organized as follows: In Section 2, we recall some facts about Dirac-harmonic maps as well as the Dirac operator. In Section 3, we prove Theorem \ref{inv} and end this section by showing the density of minimal kernel. In Section 4, under the minimality assumption on the kernel of $\slashed{D}_{1,0}^{u_0}$, we prove the short time existence of $\alpha$-Dirac-harmonic map flow (Theorem \ref{short-time alpha flow}) and the existence of Dirac-harmonic maps (Theorem \ref{flow to DH maps}). In the Appendix,  we solve the constraint equation and prove Lipschitz continuity of the solution with respect to the map.

\vspace{2em}

\section{Preliminaries}

Let $(M, g)$ be a compact Riemann surface with a fixed spin structure $\chi$. On the complex spinor bundle $\Sigma M$, we denote the Hermitian inner product by $\langle\cdot, \cdot\rangle_{\Sigma M}$. For any $X\in\Gamma(TM)$ and $\xi\in\Gamma(\Sigma M)$, the Clifford multiplication satisfies the following skew-adjointness:
\begin{equation*}
\langle X\cdot\xi, \eta\rangle_{\Sigma M}=-\langle\xi, X\cdot\eta\rangle_{\Sigma M}.
\end{equation*}
Let $\nabla$ be the Levi-Civita connection on $(M,g)$. There is a unique  connection (also denoted by $\nabla$) on $\Sigma M$ compatible with $\langle\cdot, \cdot\rangle_{\Sigma M}$.  Choosing a local orthonormal basis $\{e_{\beta}\}_{\beta=1,2}$ on $M$, the usual Dirac operator is defined as $\slashed\partial:=e_\beta\cdot\nabla_\beta$, where $\beta=1,2$. Here and in the sequel, we use the Einstein summation convention. One can find more about spin geometry in \cite{lawson1989spin}.

Let $u$ be a smooth map from $M$ to another compact Riemannian manifold $(N, h)$ of dimension $n\geq2$. Let $u^*TN$ be the pull-back bundle of $TN$ by $u$ and consider the twisted bundle $\Sigma M\otimes_{\mathbb R} u^*TN$. On this bundle there is a metric $\langle\cdot,\cdot\rangle_{\Sigma M\otimes u^*TN}$ induced from the metric on $\Sigma M$ and $u^*TN$. Also, we have a connection $\tilde\nabla$ on this twisted bundle naturally induced from those on $\Sigma M$ and $u^*TN$. In local coordinates $\{y^i\}_{i=1,\dots,n}$, the section $\psi$ of $\Sigma M\otimes_{\mathbb R} u^*TN$ is written as
$$\psi=\psi_i\otimes\partial_{y^i}(u),$$
where each $\psi^i$ is a usual spinor on $M$. We also have the following local expression of $\tilde\nabla$
$$\tilde\nabla\psi=\left(\nabla\psi^i+\Gamma_{jk}^i(u)\nabla u^j\psi^k\right)\otimes\partial_{y^i}(u),$$
where $\Gamma^i_{jk}$ are the Christoffel symbols of the Levi-Civita connection of $N$. The Dirac operator along the map $u$ is defined as
\begin{equation}\label{dirac operator}
\slashed{D}:=e_\alpha\cdot\tilde\nabla_{e_\alpha}\psi=\left(\slashed\partial\psi^i+\Gamma_{jk}^i(u)\nabla_{e_\alpha}u^j(e_\alpha\cdot\psi^k)\right)\otimes\partial_{y^i}(u),
\end{equation}
which is self-adjoint \cite{jost2017riemannian}. Sometimes, we use $\slashed{D}^u$ to distinguish the Dirac operators defined on different maps. In \cite{chen2006dirac}, the authors  introduced the  functional
\begin{equation*}\begin{split}
L(u,\psi)&:=\frac12\int_M\left(|du|^2+\langle\psi,\slashed{D}\psi\rangle_{\Sigma M\otimes u^*TN}\right)\\
&=\frac12\int_M\left( h_{ij}(u)g^{\alpha\beta}\frac{\partial u^i}{\partial x^\alpha}\frac{\pt u^j}{\pt x^\beta}+h_{ij}(u)\langle\psi^i,\slashed{D}\psi^j\rangle_{\Sigma M}\right).
\end{split}
\end{equation*}
They computed the Euler-Lagrange equations of $L$:
\begin{numcases}{}
   \tau^m(u)-\frac12R^m_{lij}\langle\psi^i,\nabla u^l\cdot\psi^j\rangle_{\Sigma M}=0,  \label{eldh1}\\
   \slashed{D}\psi^i:=\slashed\pt\psi^i+\Gamma_{jk}^i(u)\nabla_{e_\alpha}u^j(e_\alpha\cdot\psi^k)=0, \label{eldh2}
\end{numcases}
where $\tau^m(u)$ is the $m$-th component of the tension field \cite{jost2017riemannian} of the map $u$ with respect to the coordinates on $N$, $\nabla u^l\cdot\psi^j$ denotes the Clifford multiplication of the vector field $\nabla u^l$ with the spinor $\psi^j$, and $R^m_{lij}$ stands for the components of the Riemann curvature tensor of the target manifold $N$. Denote
$$\mathcal{R}(u,\psi):=\frac12R^m_{lij}\langle\psi^i,\nabla u^l\cdot\psi^j\rangle_{\Sigma M}\pt_{y^m}.$$ We can write \eqref{eldh1} and \eqref{eldh2} in the following global form:
\begin{numcases}{}
\tau(u)=\mathcal{R}(u,\psi), \label{geldh1} \\
\slashed{D}\psi=0,  \label{geldh2}
\end{numcases}
and call the solutions $(u,\psi)$ Dirac-harmonic maps from $M$ to $N$.

With the aim to get a general existence scheme for  Dirac-harmonic maps, the following heat flow for Dirac-harmonic maps was introduced in \cite{chen2017estimates}:
\begin{empheq}[left=\empheqlbrace]{alignat=2}
	&\partial_tu=\tau(u)-\mathcal{R}(u,\psi),\ &\text{on} \   (0,T)\times M, \label{map} \\
		&\slashed{D}\psi=0,\ & \text{on} \   [0,T]\times M.	\label{dirac}
\end{empheq}
When $M$ has boundary, the short time existence and uniqueness of \eqref{map}-\eqref{dirac} was  shown in \cite{chen2017estimates}.

 For a closed manifold $M$, the situation is  more complicated because one cannot uniquely solve the second equation \eqref{dirac} and the kernel of the Dirac operator may jump along the flow. As we stated in the introduction, the short-time existence is only known in the minimal kernel case, i.e. ${\rm dim}_{\mathbb H}{\rm ker}\slashed{D}=1$. However, when the target manifold $N$ is an even-dimensional spin manifold, the index $\alpha(M,u)$ always vanishes for any map $u$ between $M$ and $N$. In order to deal with this case, we utilize the complex structure on $N$. We denote the complexification of $u^*TN$ by $(u^*TN)^{\mathbb C}$. Then we have
 \begin{equation*}
 \Sigma M\otimes_{\mathbb R} u^*TN=(\Sigma M\otimes_{\mathbb C}{\mathbb C})\otimes_{\mathbb R} u^*TN=\Sigma M\otimes_{\mathbb C}(u^*TN)^{\mathbb C}.
\end{equation*}
 The pull-back metric $u^*g$ on $u^*TN$ could be naturally extended to a Hermitian
product on $(u^*TN)^{\mathbb C}$. Moreover, there is a natural Hermitian product on $\Sigma M\otimes(u^*TN)^{\mathbb C}$ induced from those on $\Sigma M$ and $(u^*TN)^{\mathbb C}$, which is denoted by $\langle\cdot,\cdot\rangle_{\Sigma M\otimes(u^*TN)^{\mathbb C}}$.
\iffalse
Generally, \eqref{vector bundle identity} holds for any real vector bundle $E$ over $M$, i.e.
 \begin{equation*}
 \Sigma M\otimes_{\mathbb R}E=\Sigma M\otimes_{\mathbb C} E^{\mathbb C},
\end{equation*}
 where $E^\mathbb{C}$ is the complexification of $E$.
\fi

For a general even-dimensional spin Riemannian manifold $M$, there is a parallel ${\mathbb Z}_2$-grading $G\in{\rm End}(\Sigma M)$ given by $G(\psi)=\left(\sqrt{-1}\right)^{m/2}e_1\cdot e_2\cdot\dotsc\cdot e_{m}\cdot\psi$ for a positively oriented orthonormal local
frame $\{e_1,e_2,\dotsc,e_m\}$ where $m=\dim M$. Thus the spinor bundle can be decomposed as
\begin{equation*}
\Sigma M=\Sigma^+M\oplus\Sigma^-M,
\end{equation*}
where $\Sigma^\pm M$ are the eigenspaces of $G$ associated to the $\pm1$, respectively. As $G$ is Hermitian and parallel, the decomposition is orthogonal in the complex sense and parallel. Consequently, we have
\begin{equation}\label{decomposition}
\begin{split}
\Sigma M\otimes_{\mathbb C}(u^*TN)^{\mathbb C}=&(\Sigma^+M\otimes_{\mathbb C}u^*T_{1,0}N)\oplus(\Sigma^-M\otimes_{\mathbb C}u^*T_{1,0}N)\\
&\oplus(\Sigma^+M\otimes_{\mathbb C}u^*T_{0,1}N)\oplus(\Sigma^-M\otimes_{\mathbb C}u^*T_{0,1}N),
\end{split}
\end{equation}
 where we used $(u^*TN)^{\mathbb C}=u^*T_{1,0}N\oplus u^*T_{1,0}N$. Moreover, we also have the following decomposition for the  Dirac operator:
 \begin{equation*}
 \begin{split}
 \slashed{D}=&\slashed{D}_{1,0}+\slashed{D}_{0,1}\\
 \slashed{D}_{1,0}=&\slashed{D}_{1,0}^++\slashed{D}_{1,0}^-\\
  \slashed{D}_{0,1}=&\slashed{D}_{0,1}^++\slashed{D}_{0,1}^-,
 \end{split}
\end{equation*}
 where $\slashed{D}_{1,0}^\pm$ (resp. $\slashed{D}_{0,1}^\pm$) is obtained by restricting $\slashed{D}$ on $\Sigma^\pm M\otimes_{\mathbb C}u^*T_{1,0}N$ (resp. $\Sigma^\pm M\otimes_{\mathbb C}u^*T_{0,1}N$).

By \cite{nash1956imbedding}, we can isometrically embed $N$ into $\mathbb{R}^q$. Then \eqref{geldh1}-\eqref{geldh2} is equivalent to the  following system:
\begin{equation*}
    \begin{cases}
		\Delta_g{u}=II(du,du)+Re(P(\mathcal{S}(du(e_\beta),e_{\beta}\cdot\psi);\psi)), \\
		\slashed{\partial}\psi=\mathcal{S}(du(e_\beta),e_{\beta}\cdot\psi),
	\end{cases}
\end{equation*}
where $II$ is the second fundamental form of $N$ in $\mathbb{R}^q$, and
\begin{equation*}
\mathcal{S}(du(e_\beta),e_{\beta}\cdot\psi):=(\nabla{u^A}\cdot\psi^B)\otimes II(\partial_{z^A},\partial_{z^B}),
\end{equation*}
\begin{equation*}
Re(P(\mathcal{S}(du(e_\beta),e_{\beta}\cdot\psi);\psi)):=P(S(\partial_{z^C},\partial_{z^B});\partial_{z^A})Re(\langle\psi^A,du^C\cdot\psi^B\rangle).
\end{equation*}
Here $P(\xi;\cdot)$ denotes the shape operator, defined by $\langle P(\xi;X),Y\rangle=\langle A(X,Y),\xi\rangle$ for $X,Y\in\Gamma(TN)$ and $Re(z)$ denotes the real part of $z\in\mathbb{C}$. Together with the \emph{nearest point projection}:
\begin{equation*}
\pi: \ N_{\delta}\to N,
\end{equation*}
where $N_{\delta}:=\{z\in\mathbb{R}^q| d(z,N)\leq\delta\}$, we can rewrite the evolution equation \eqref{map} as an equation in $\mathbb{R}^q$.

\begin{lem}\cite{chen2017estimates}
A tuple $(u,\psi)$, where $u:[0,T]\times M\to N$ and $\psi\in \Gamma(\Sigma M\otimes u^*TN)$, is a solution of \eqref{map} if and only if
\begin{equation*}
\partial_tu^A-\Delta u^A=-\pi^A_{BC}(u)\langle\nabla u^B,\nabla u^C\rangle-\pi^A_B(u)\pi^C_{BD}(u)\pi^C_{EF}(\psi^D,\nabla u^E\cdot\psi^F)
\end{equation*}
on $(0,T)\times M$, for $A=1,\dots,q$. Here we denote the $A$-th component function of $u: [0,T]\times M\to N\subset\mathbb{R}^q$ by $u^A: M\to\mathbb{R}$, write $\pi^A_B(z)$ for the $B$-th partial derivative of the $A$-th component function of $\pi: \mathbb{R}^q\to\mathbb{R}^q$ and the global sections $\psi^A\in\Gamma(\Sigma M)$ are defined by $\psi=\psi^A\otimes(\partial_A\circ u)$, where $(\partial_A)_{A=1,\dots,q}$ is the standard basis of $T\mathbb{R}^q$. Moreover, $\nabla$ and $\langle\cdot,\cdot\rangle$ denote the gradient and the Riemannian metric on $M$, respectively.
\end{lem}

For future reference, we define
\begin{equation}\label{F1A}
F_1^A(u):=-\pi^A_{BC}(u)\langle\nabla u^B,\nabla u^C\rangle,
\end{equation}
\begin{equation}\label{F2A}
F_2^A(u,\psi):=-\pi^A_B(u)\pi^C_{BD}(u)\pi^C_{EF}(\psi^D,\nabla u^E\cdot\psi^F).
\end{equation}
Note that for $u\in C^1(M,N)$ and $\psi\in\Gamma(\Sigma M\otimes u^*TN)$ we have
\begin{equation*}
II(du_p(e_\alpha), du_p(e_\alpha)))=-F_1^A(u)|_p\partial_A|_{u(p)},
\end{equation*}
\begin{equation*}
\mathcal{R}(u,\psi)|_p=-F_2^A(u,\psi)|_p\partial_A|_{u(p)}
\end{equation*}
for all $p\in M$, where $\{e_\alpha\}$ is an orthonormal basis of $T_pM$.

Next, let us fix some notations, which will be used in the Section 4 and Appendix. For every $T>0$, we denote by $X_T$ the Banach space of bounded maps:
\begin{equation*}
X_T:=B([0,T]; C^1(M,\mathbb{R}^q)),
\end{equation*}
\begin{equation*}
\|u\|_{X_T}:=\max\limits_{A=1,\dots,q}\sup\limits_{t\in[0,T]}(\|u^A(t,\cdot)\|_{C^0(M)}+\|\nabla u^A(t,\cdot)\|_{C^0(M)}).
\end{equation*}
For any map $v\in X_T$, the closed ball with center $v$ and radius $R$ in $X_T$ is defined by
\begin{equation*}
B_R^T(v):=\{u\in X_T|\|u-v\|\leq R\}.
\end{equation*}
We denote by $P^{u_t,v_s}=P^{u_t,v_s}(x)$ the parallel transport of $N$ along the unique shortest geodesic from $\pi(u(x,t))$ to $\pi(v(x,s))$. We also denote by $P^{u_t,v_s}$ the inducing  mappings
\begin{equation*}
(\pi\circ u_t)^*TN\to(\pi\circ v_s)^*TN,
\end{equation*}
\begin{equation*}
\Sigma M\otimes(\pi\circ u_t)^*TN\to\Sigma M\otimes(\pi\circ v_s)^*TN
\end{equation*}
and
\begin{equation*}
\Gamma_{C^1}(\Sigma M\otimes(\pi\circ u_t)^*TN)\to\Gamma_{C^1}(\Sigma M\otimes(\pi\circ v_s)^*TN).
\end{equation*}
We also define
\begin{equation*}
\Lambda(u_t)=\sup\{\tilde\Lambda| {\rm spec}(\slashed{D}^{\pi\circ u_t})\setminus\{0\}\subset\mathbb{R}\setminus(-\tilde\Lambda(u_t),\tilde\Lambda(u_t))\}
\end{equation*}
 and $\gamma_t(x): [0,2\pi]\to\mathbb{C}$ as
 \begin{equation*}
 \gamma_t(x):=\frac{\Lambda(u_t)}{2}e^{ix}.
 \end{equation*}
 In general, we also denote by $\gamma$ the curve $\gamma(x): [0,2\pi]\to\mathbb{C}$ as
 \begin{equation}\label{general curve}
 \gamma(x):=\frac{\Lambda}{2}e^{ix}
 \end{equation}
 for some constant $\Lambda$ to be determined. Then the orthogonal projection onto ${\rm ker}(\slashed{D}^{\pi\circ u_t})$, which is the mapping
 \begin{equation*}
\Gamma_{L^2}(\Sigma M\otimes(\pi\circ u_t)^*TN)\To\Gamma_{L^2}(\Sigma M\otimes(\pi\circ u_t)^*TN),
\end{equation*}
can be written via the resolvent by
\begin{equation*}
s\mapsto-\frac{1}{2\pi i}\int_{\gamma_t}R(\lambda,\slashed{D}^{\pi\circ u_t})sd\lambda,
\end{equation*}
where $R(\lambda,\slashed{D}^{\pi\circ u_t}): \Gamma_{L^2}\to\Gamma_{L^2}$ is the resolvent of $\slashed{D}^{\pi\circ u_t}: \Gamma_{W^{1,2}}\to\Gamma_{L^2}$.

In the end of this section, we recall the definition of the index.

\begin{defn}\label{index alpha}
 Let $E\to M$ be a Riemannian real vector bundle with metric connection. Then one
can associate to the twisted Dirac operator $\slashed{D}^E: C^\infty(M,\Sigma M\otimes E)\to C^\infty(M,\Sigma M\otimes E)$ an index $I(M,\chi,E)\in KO_m({\rm pt})$, where
\begin{equation*}
		KO_m({\rm pt})\cong\begin{cases}
		\mathbb{Z}, & \text{if} \ m=0\ (4), \\
		\mathbb{Z}_2,& \text{if} \ m=1,2\ (8),\\
		0, & \text{otherwise}.
		\end{cases}
	\end{equation*}
The index $I(M,\chi,E)$ can be determined out of ${\rm ker}(\slashed{D}^E)$ using the following formula:
\begin{equation*}
	I(M,\chi,E)=\begin{cases}
		\{{\rm ch}(E)\cdot\hat{A}(M)\}[M], & \text{if} \ m=0\ (8), \\
		[{\rm dim}_\mathbb{C}({\rm ker}(\slashed{D}^E)]_{\mathbb{Z}_2},& \text{if} \ m=1\ (8),\\
		\left[\frac{{\rm dim}_\mathbb{C}({\rm ker}(\slashed{D}^E)}{2}\right]_{\mathbb{Z}_2},& \text{if} \ m=2\ (8),\\
		\frac12\{{\rm ch}(E)\cdot\hat{A}(M)\}[M], & \text{if} \ m=4\ (8).
		\end{cases}
	\end{equation*}

In particular, when $E=u^*TN$ and $\chi$ is fixed, we denote $I(M,\chi,E)$ by $I(M,u)$.
\end{defn}

\vspace{2em}
\section{Quaternionic structure on the twisted bundle }
In this section, we will prove Theorem \ref{inv} by constructing a commuting quaternionic structure on the twisted bundle $\Sigma M\otimes_{\mathbb{C}}u^*T_{1,0}N$ and show the density of the minimal kernel.

\begin{proof}[Proof of Theorem \ref{inv}]
Let $\rho: \mathbb{C}{\rm l}_m\to{\rm End}_{\mathbb{C}}(\Sigma_m)$ be an irreducible complex  representation of the complex Clifford algebra $\mathbb{C}{\rm l}_m$.  Suppose the condition (2) holds. Then every fibre of the complex spinor bundle $\Sigma M={\rm Spin}(M)\times_\rho\Sigma_m$ turns into a quaternionic vector space by defining
\begin{equation*}
[p,v]h:=[p,vh]
\end{equation*}
for all $p\in{\rm Spin}(M)$, $v\in\Sigma_m$ and $h\in\mathbb{H}$.

Since the tensor product of the twisted bundle $\Sigma M\otimes_{\mathbb{C}}u^*T_{1,0}N$ is taken over $\mathbb{C}$, there is a natural complex structure $I$ on $\Sigma M\otimes_{\mathbb{C}}u^*T_{1,0}N$ defined by
\begin{equation*}
I(\psi^k\otimes_{\mathbb{C}}\theta_k):=i(\psi^k)\otimes_{\mathbb{C}}\theta_k=\psi^k\otimes_{\mathbb{C}}i(\theta_k).
\end{equation*}
However, the quaternionic structure on $\Sigma M$ cannot directly extend to the twisted bundle. To overcome this problem, we need an extra structure on $u^*T_{1,0}N$. By our assumption, we define $J: \Sigma M\otimes_{\mathbb{C}} u^*T_{1,0}N\to\Sigma M\otimes_{\mathbb{C}} u^*T_{1,0}N$ by
\begin{equation*}
J(\psi^i\otimes_{\mathbb{C}}\theta_i):=j_1(\psi^k)\otimes_{\mathbb{C}}j_2(\theta_k).
\end{equation*}
Since both $j_1$ and $j_2$  anti-commute with the complex structure $i$, $J$ is well-defined on $\Sigma M\otimes_{\mathbb{C}} u^*T_{1,0}N$. By the definitions of $j_1$ and $j_2$, $J$  anti-commutes with $I$ and $J^2=-1$.  Moreover, $J$  also commutes with the Clifford multiplication and hence the Dirac operator $\slashed{D}^u_{1,0}$ (see also \cite{hermann2012dirac}), i.e.
\begin{equation*}
\slashed{D}^u_{1,0}\circ J=J\circ\slashed{D}^u_{1,0}.
\end{equation*}
Therefore, we conclude that all the eigenspaces of $\slashed{D}^u_{1,0}$ are quaternionic vector space with two complex structures $I$ and $J$, which are anti-commuting with each other.

If condition (1) holds, i.e. $j_1$ is a commuting real structure and $j_2$ is a quaternionic structure, then  it follows from the argument above that the conclusion is also true.

When $m\neq3({\rm mod} \ 4)$, the eigenvalues are symmetric with respect to the origin (see Remark 2.2.3 in \cite{hermann2012dirac}). For any two maps in $[u]$, there is a piecewise smooth curve connecting them with parameter $t\in[0,1]$. Along this curve, the eigenvalues of the Dirac operator are continuous functions of $t$. Suppose there is an eigenvalue $\lambda_1(t)$ that decreases to zero as $t\to T$. By the symmetry of the  eigenvalues, there is another eigenvalue $\lambda_{-1}(t)$ such that $\lambda_{-1}(t)=-\lambda_{1}(t)$. Therefore, the difference of the quaternionic dimension of the kernel of the corresponding Dirac operator is always an even number.

When $m$ is even, we have a parallel $\mathbb{Z}_2$-grading $G$ described in the previous section. From the orthogonality of the splitting, we have
\begin{equation*}
\langle\slashed{D}^{u}_{1,0}\psi^+,\psi^+\rangle=\langle\slashed{D}^u_{1,0}\psi^-,\psi^-\rangle=0
\end{equation*}
for all $\psi^\pm\in C^{\infty}(M, \Sigma^{\pm}M\otimes u^*T_{1,0}N)$. Thus,
\begin{equation}\label{zero L2}
(\slashed{D}^{u}_{1,0}\psi^+,\psi^+)_{L^2}=(\slashed{D}^{u}_{1,0}\psi^-,\psi^-)_{L^2}=0.
\end{equation}

Now, for any smooth variation $(u_s)_{s\in(-\epsilon,\epsilon)}$ of the $\alpha$-harmonic map $\tilde{u}\in[u]$ with $u_s|_{s=0}=\tilde{u}$, we split the bundle $\Sigma M\otimes u_s^*T_{1,0}N$ into
\begin{equation*}
\Sigma M\otimes u_s^*T_{1,0}N=(\Sigma^+M\otimes u_s^*T_{1,0}N)\oplus(\Sigma^-M\otimes u_s^*T_{1,0}N),
\end{equation*}
 which is orthogonal in the complex sense and parallel. Since $\mathcal{I}(M, u^*T_{1,0}N)\neq0$, there exists $\Psi\in{\rm ker}{\slashed{D}^{\tilde{u}}_{1,0}}$ which can be written as $\Psi=\Psi^++\Psi^-$, where $\Psi^{\pm}\in\Gamma(\Sigma^{\pm}M\otimes \tilde{u}^*T_{1,0}N)$. Then there always exists a variation $\Psi_s$ of $\Psi$ such that $\Psi_s^{\pm}\in\Gamma(\Sigma^{\pm}M\otimes u_s^*T_{1,0}N)$  are smooth variations of $\Psi^\pm$, respectively. Moreover, \eqref{zero L2} implies that
\begin{equation*}
\frac{d}{dt}\bigg|_{s=0}(\slashed{D}^{u_s}\Psi_s^\pm,\Psi_s^\pm)_{L^2}=0.
\end{equation*}
Therefore, for the $\alpha$-harmonic map $\tilde{u}$, we have 
\begin{equation*}
\frac{d}{dt}\bigg|_{s=0}L^\alpha(u_s,\Psi_s^{\pm})=\frac{d}{dt}\bigg|_{s=0}\int_{M}(1+|du_s|^2)^\alpha=0.
\end{equation*}
Hence, we get $\alpha$-Dirac-harmonic maps $(\tilde{u},\Psi^\pm)$.

\end{proof}

In the rest of this section, we will show the density of the minimal kernel. By the definition of $\mathcal{I}(M, u^*T_{1,0}N)$, we have
\begin{equation*}
{\rm dim}_{\mathbb H}{\rm ker}(\slashed{D}_{1,0}^{u})\geq
\begin{cases}
0,& \text{if} \ {\rm ind}_{u^*T_{1,0}N}(M)=0;  \\
1,& \text{if} \ {\rm ind}_{u^*T_{1,0}N}(M)\neq0.
\end{cases}
\end{equation*}
If equality holds above, then we say that $\slashed{D}_{1,0}^{u}$ has minimal kernel. Using the analyticity of $N$, one can prove the following density result for the  minimal kernel.

\begin{lem}\label{density}
If $\slashed{D}_{1,0}^{u}$ has minimal kernel, then  $\slashed{D}_{1,0}^{u'}$ also has minimal kernel for a generic map $u'\in[u]$.
\end{lem}

\begin{proof}
Let $u'\in[u]$ and $H$ be any homotopy between $u'$ and $u$. More precisely, $H: [0,1]\to C^\infty(M,N)$ with $H(0)=u$ and $H(1)=u'$. We can cover the image of $H$ by finitely many balls $\{V_l\}_{l=1}^L$ of radius less than $\frac12{\rm inj}(N)$ such that
\begin{equation*}
V_l\cap V_{l+1}\neq\emptyset \ \text{for} \ i=1,\cdots,L-1
\end{equation*}
and
\begin{equation*}
u\in V_1, \ \ \ u'\in V_L.
\end{equation*}

We choose $u_1\in V_1\cap V_2$ arbitrarily and define a homotopy $H^1_t$ by
\begin{equation*}
H^1_t(x):=\exp_{u(x)}(t\exp^{-1}_{u(x)}u_1(x)),
\end{equation*}
where $x\in M$ and $\exp$ is the exponential map on $N$. We denote by
\begin{equation*}
P_{t}=P_t(x): T_{1,0}N|_{u(x)}\to T_{1,0}N|_{H^1_t(x)}
\end{equation*}
the parallel transport along the unique shortest geodesic of $N$ connecting $u(x)$ and $H^1_t(x)$ and consider
\begin{equation*}
\slashed{D}_t:=P_t^{-1}\circ\slashed{D}_{1,0}^{H^1_t}\circ P_t.
\end{equation*}
Since $\slashed{D}_t$ depends analytically on $t$ by the analyticity of $N$, $\slashed{D}^{u_t}_{1,0}$ has minimal kernel for all but finitely many $t\in[0,1]$. Therefore, we can assume $\slashed{D}^{u_1}_{1,0}$ has minimal kernel. Continuing this procedure, we can get $u_{L-1}\in V_{L-1}\cap V_{L}$ such that $\slashed{D}^{u_{L-1}}_{1,0}$ also has minimal kernel and a homotopy $H^{L-1}_t$ between $u_{L-1}$ and $u'$ such that $\slashed{D}^{H^{L-1}_t}_{1,0}$ has minimal kernel for all but finitely many $t\in[0,1]$. Hence the set of maps along which the $(1,0)$-part of the Dirac operator has minimal kernel is $C^\infty$-dense in $[u]$. Its $C^1$-openness directly follows from the continuity of the eigenvalues.

\end{proof}

\vspace{2em}
\section{The heat flow for \texorpdfstring{$\alpha$}{alpha}-Dirac-harmonic maps}
In this section, we will prove the short-time existence of the heat flow for $\alpha$-Dirac-harmonic maps. Since we are working on a closed surface $M$, we cannot uniquely solve the Dirac equation in the following system:
\begin{numcases}{}
		\partial_tu=\frac{1}{(1+|\nabla u|^2)^{\alpha-1}}\bigg(\tau^\alpha(u)-\frac{1}{\alpha}\mathcal{R}(u,\psi)\bigg), \label{alpha map}\\
		\slashed{D}^{u}\psi=0.	\label{alpha dirac}
	\end{numcases}
 The short time existence and its extension are the obstacles. This system (if it converges) leads to an $\alpha$-Dirac-harmonic map which is a solution of the system
 \begin{align*}
\begin{cases}
\tau^\alpha(u):=\tau((1+|du|^2)^\alpha)=\frac{1}{\alpha}\mathcal{R}(u,\psi),\\
\slashed{D}^u\psi=0,
\end{cases}
\end{align*}
 where $\tau$ is the tension field.

\subsection{Short time existence}

As in Section 2, we now embed $N$ into $\mathbb{R}^q$. Let $u: M\to N$ with $u=(u^A)$ and denote the spinor along the map $u$ by $\psi=\psi^A\otimes(\partial_A\circ u)$, where $\psi^A$ are spinors over $M$. For any smooth map $\eta\in C^\infty_0(M,\mathbb{R}^q)$ and any smooth spinor field $\xi\in C^\infty_0(\Sigma M\otimes\mathbb{R}^q)$, we consider the variation
\begin{equation}
u_t=\pi(u+t\eta), \ \ \ \psi^A_t=\pi^A_B(u_t)(\psi^B+t\xi^B),
\end{equation}
where $\pi$ is the nearest point projection as in Section 2. Then we have
\begin{lem}[\cite{jost2019short}]\label{alpha EL}
The Euler-Lagrange equations for $L^\alpha$ are
\begin{equation*}
\begin{split}
\Delta u^A&=-2(\alpha-1)\frac{\nabla^2_{\beta\gamma}u^B\nabla_{\beta}u^B\nabla_{\gamma}u^A}{1+|\nabla u|^2}+\pi^A_{BC}(u)\langle\nabla{u^B},\nabla{u^C}\rangle\\
&\quad+\frac{\pi^A_B(u)\pi^C_{BD}(u)\pi^C_{EF}(u)\langle\psi^D,\nabla{u}^E\cdot\psi^F\rangle}{\alpha(1+|\nabla{u}|^2)^{\alpha-1}}
\end{split}\end{equation*}
and
\begin{equation*}
\slashed{\partial}\psi^A=\pi^A_{BC}(u)\nabla{u}^B\cdot\psi^C.
\end{equation*}
\end{lem}

 Lemma \ref{alpha EL} implies that \eqref{alpha map}-\eqref{alpha dirac} is equivalent to
\begin{numcases}{}
\begin{aligned}
		\partial_tu^A&=\Delta{u^A}+2(\alpha-1)\frac{\nabla^2_{\beta\gamma}u^B\nabla_{\beta}u^B\nabla_{\gamma}u^A}{1+|\nabla u|^2}-\pi^A_{BC}(u)\langle\nabla{u^B},\nabla{u^C}\rangle \label{system in Euclidean1}\\
		&\quad-\frac{\pi^A_B(u)\pi^C_{BD}(u)\pi^C_{EF}(u)\langle\psi^D,\nabla{u}^E\cdot\psi^F\rangle}{\alpha(1+|\nabla{u}|^2)^{\alpha-1}},
		\end{aligned}\\
		\slashed{D}^{\pi\circ u}\psi=0,	\label{system in Euclidean2}
	\end{numcases}

Now, let us state the main result of this subsection.

\begin{thm}\label{short-time alpha flow}
Let $M$ be a closed surface, and $N$ a closed $n$-dimensional Riemannian manifold. Let $u_0\in C^{2+\mu}(M,N)$ for some $0<\mu<1$ with ${\rm dim}_{\mathbb H}{\rm ker}(\slashed{D}_{1,0}^{u_0})=1$ and $\psi_0\in{\rm ker}(\slashed{D}_{1,0}^{u_0})$ with $\|\psi_0\|_{L^2}=1$. Then there exists $\epsilon_1=\epsilon_1(M,N)>0$ such that, for any $\alpha\in(1,1+\epsilon_1)$, the problem \eqref{alpha map}-\eqref{alpha dirac} has a solution $(u,\psi)$ with
\begin{equation}\label{initial value}
		\begin{cases}
		 \|\psi_t\|_{L^2}=1, & \ \forall t\in[0,T],\\
		u|_{t=0}=u_0, \  \psi|_{t=0}=\psi_0.
		\end{cases}
	\end{equation}
 satisfying
\begin{equation*}
u\in C^{2+\mu,1+\mu/2}(M\times[0,T],N)
\end{equation*}
and
\begin{equation*}
\psi\in C^{\mu,\mu/2}(M\times[0,T], \Sigma M\otimes u^*TN)\cap L^\infty([0,T];C^{1+\mu}(M)).
\end{equation*}
for some $T>0$.
\end{thm}

\begin{proof}
We will prove the theorem in two steps. In Step 1, we will find a solution $u: M\times[0,T]\to\mathbb{R}^q$ and $\psi_t: M\to\Sigma M\otimes(\pi\circ{u_t})^*TN$ of \eqref{system in Euclidean1}-\eqref{system in Euclidean2} with the initial values \eqref{initial value}. Since $\psi_t$ takes the value along the projection $\pi\circ{u_t}$, it remains to show $u$ takes the value in $N$, which will be proved in Step 2.

{\bf Step 1:} Solving \eqref{system in Euclidean1}-\eqref{system in Euclidean2} in $\mathbb{R}^q$.

   We first give a solution to \eqref{system in Euclidean2} in a neighborhood of $u_0$. For any $T>0$, we can choose $\epsilon$, $\delta$ and $R$ as in the Appendix such that
\begin{equation*}
u(x,t)\in N_\delta
\end{equation*}
and
\begin{equation*}
d^N((\pi\circ u)(x,t),(\pi\circ v)(x,s))<\epsilon<\frac12{\rm inj}(N)
\end{equation*}
for all $u,v\in B^T_R:=B^T_R(\bar{u}_0)=\{u\in X_T|\|u-\bar{u}_0\|_{X_T}\leq R\}\cap\{u|_{t=0}=u_0\}$, $x\in M$ and $t,s\in[0,T]$, where $\bar{u}_0(x,t)=u_0(x)$ for any $t\in[0,T]$. If $R$ is small enough, then by Lemma \ref{dim of kernel u}, we have
\begin{equation*}
{\rm dim}_{\mathbb{H}}{\rm ker}(\slashed{D}_{1,0}^{\pi\circ u_t})=1
\end{equation*}
and there exists $\Lambda=\frac12\Lambda(u_0)$ such that
\begin{equation*}
\#\{{\rm spec}(\slashed{D}_{1,0}^{\pi\circ u_t})\cap[-\Lambda,\Lambda]\}=1
\end{equation*}
for any $u\in B_R^T$ and $t\in[0,T]$, where $\Lambda(u_0)$ is a constant such that ${\rm spec}(\slashed{D}_{1,0}^{u_0})\setminus\{0\}\subset\mathbb{R}\setminus[-\Lambda(u_0),\Lambda(u_0)]$. Furthermore, for $\psi_0\in{\rm ker}(\slashed{D}_{1,0}^{u_0})$ with $\|\psi_0\|_{L^2}=1$, Lemma \ref{projection to kernel} implies that
\begin{equation*}
\sqrt{\frac34}\leq\|\tilde\psi_1^{u_t}\|_{L^2}\leq1
\end{equation*}
for any $u\in B_{R_1}^{T}$ and $t\in[0,T]$, where $\tilde\psi^{u_t}=P^{u_0,u_t}\psi=\tilde\psi^{u_t}_1+\tilde\psi^{u_t}_2$ with respect to the decomposition $\Gamma_{L^2}={\rm ker}(\slashed{D}_{1,0}^{\pi\circ u_t})\oplus({\rm ker}(\slashed{D}_{1,0}^{\pi\circ u_t}))^\bot$ and $R_1=R_1(R,\epsilon,u_0)>0$.

Now, for any $T>0$ and $\kappa>0$, we define
$$V^T_{\kappa}:=\{v\in C^{1+\mu,\frac{1+\mu}{2}}(M\times[0,T])|\|v\|_{C^{1+\mu,\frac{1+\mu}{2}}}\leq\kappa, \ v|_{M\times\{0\}}=0\}.$$
Then, there exists $\kappa_{R_1}:=\kappa(R_1)>0$ such that
\begin{equation*}
u_0+v\in B^{T}_{R_1}, \ \forall v\in V_{\kappa}^T, \ \forall\kappa\leq\kappa_{R_1}.
\end{equation*}
Now, we denote $\kappa_0:=\kappa_{R_1}$ and $V^T:= V^T_{\kappa_0}$.

For every $v\in V^T$, $u_0+v\in B^{T}_{R_1}$, Lemma \ref{lip} gives us a solution $\psi(v+u_0)$ to the  constraint equation. Since $v+u_0\in C^{1+\mu}(M)$, by $L^p$ regularity and Schauder estimates in  \cite{chen2017estimates}, we have
\begin{equation}\label{C1+mu}
\|\psi(v+u_0)\|_{C^{1+\mu}(M)}\leq C(\mu, M, N, \kappa_0, \|u_0\|_{C^{1+\mu}(M)}).
\end{equation}

For any $0<t,s<T$, we also have
\begin{equation*}
\begin{split}
&\quad\slashed\partial(\psi(v+u_0)(t)-\psi(v+u_0)(s))\\
&=-\Gamma(\pi\circ(v+u_0)(t))\#\nabla(\pi\circ(v+u_0)(t))\#\psi(v+u_0)(t)\\
&\quad+\Gamma(\pi\circ(v+u_0)(s))\#\nabla(\pi\circ(v+u_0)(s))\#\psi(v+u_0)(s)\\
&=-\Gamma(\pi\circ(v+u_0)(t))\#\nabla(\pi\circ(v+u_0)(t))\#(\psi^v(t)-\psi(v+u_0)(s))\\
&\quad-\Gamma(\pi\circ(v+u_0)(t))\#(\nabla(\pi\circ(v+u_0)(t))-\nabla(\pi\circ(v+u_0)(s)))\#\psi(v+u_0)(t)\\
&\quad-(\Gamma(\pi\circ(v+u_0)(t))-\Gamma(\pi\circ(v+u_0)(s)))\#\nabla(\pi\circ(v+u_0)(s))\#\psi(v+u_0)(s),
\end{split}\end{equation*}
that is,
\begin{equation*}
\begin{split}
&\quad\slashed{D}^{\pi\circ v(t)}(\psi(v+u_0)(t)-\psi(v+u_0)(s))\\
&=-\Gamma(\pi\circ(v+u_0)(t))\#(\nabla(\pi\circ(v+u_0)(t))-\nabla(\pi\circ(v+u_0)(s)))\#\psi(v+u_0)(t)\\
&\quad-(\Gamma(\pi\circ(v+u_0)(t))-\Gamma(\pi\circ(v+u_0)(s)))\#\nabla(\pi\circ(v+u_0)(s))\#\psi(v+u_0)(s),
\end{split}\end{equation*}
where $\#$ denotes a multi-linear map with smooth coefficients. For any $\lambda\in(0,1)$, by the Sobolev embedding, $L^p$ regularity and Lemma \ref{lip}, we have
\begin{equation*}
\begin{split}
&\quad\|\psi(v+u_0)(t)-\psi(v+u_0)(s)\|_{C^\lambda(M)}\\
&\leq C(\lambda, M, N, \kappa_0, \|u_0\|_{C^1(M)})(\|v(t)-v(s)\|_{L^\infty(M)}+\|dv(t)-dv(s\|_{L^\infty}))\\
&\leq C(\lambda, M, N, \kappa_0, \|u_0\|_{C^1(M)})|t-s|^{\mu/2}.
\end{split}\end{equation*}

Therefore,
\begin{equation}\label{Cmu}
\|\psi(v+u_0)\|_{C^{\mu,\mu/2}(M)}\leq C(\mu, M, N, \kappa_0, \|u_0\|_{C^1(M)}).
\end{equation}

Now, when $\alpha-1$ is sufficiently small, for the $(v,\psi^v)$ above, the standard theory of linear parabolic systems (see \cite{schlag1996schauder}) implies that there exists a unique solution $v_1\in C^{2+\mu,1+\mu/2}(M\times[0,T], \mathbb{R}^q)$ to the following Dirichlet  problem:
		\begin{equation}\label{D1}
		\begin{split}
		\partial_t w^A&=\Delta_g w^A+2(\alpha-1)\frac{\nabla^2_{\beta\gamma}w^B\nabla_{\beta}(v+u_0)^B\nabla_{\gamma}(v+u_0)^A}{1+|\nabla(v+u_0)|^2}  \\
		&\quad+\pi^A_{BC}(v+u_0)\langle\nabla{(v+u_0)^B},\nabla{(v+u_0)^C}\rangle \\
		&\quad+\frac{(\pi^A_B\pi^C_{BD}\pi^C_{EF})(v+u_0)\langle\psi^D(v+u_0),\nabla{(v+u_0)}^E\cdot\psi^F(v+u_0)\rangle}{\alpha(1+|\nabla{(v+u_0)}|^2)^{\alpha-1}}, \\
		&\quad+\Delta_gu_0^A+2(\alpha-1)\frac{\nabla^2_{\beta\gamma}u_0^B\nabla_{\beta}(v+u_0)^B\nabla_{\gamma}(v+u_0)^A}{1+|\nabla(v+u_0)|^2},
		\end{split}
		\end{equation}
	\begin{equation}\label{D2}
		w(\cdot,0)=0,	
	\end{equation}
satisfying
\begin{equation*}
\|v_1\|_{C^{2+\mu,1+\mu/2}(M\times[0,T])}\leq C(\mu,M,N)(\|v_1\|_{C^0(M\times[0,T])}+\|u_0\|_{C^{2+\nu}(M)}+\kappa_0).
\end{equation*}
Since $v_1(\cdot,0)=0$, we have
\begin{equation*}
\|v_1\|_{C^0(M\times[0,T])}\leq C(\mu,M,N)T(\|v_1\|_{C^0(M\times[0,T])}+\|u_0\|_{C^{2+\nu}(M)}+\kappa_0).
\end{equation*}
By taking $T>0$ small enough, we get
\begin{equation*}
\|v_1\|_{C^0(M\times[0,T])}\leq C(\mu,M,N)T(\|u_0\|_{C^{2+\nu}(M)}+\kappa_0).
\end{equation*}
Then the interpolation inequality in \cite{lieberman1996second} implies that $v_1\in V^T$ for $T>0$ sufficiently small. For such $v_1$, we have $\psi(v_1+u_0)$ satisfying \eqref{C1+mu} and \eqref{Cmu}. Replacing $(v,\psi(v+u_0))$ in \eqref{D1}-\eqref{D2} by $(v_1,\psi(v_1+u_0))$, then we get $v_2\in V^T$.
Iterating this procedure, we get a solution $v_{k+1}$ of \eqref{D1}-\eqref{D2} with $(v,\psi(v+u_0))$ replacing by $(v_k,\psi(v_k+u_0))$, which satisfies
\begin{equation*}
\|\psi(v_{k+1}+u_0)\|_{C^{\mu,\mu/2}(M)}\leq C(\mu, M, N, \kappa_0, \|u_0\|_{C^1(M)}).
\end{equation*}
and
\begin{equation*}
\|v_{k+1}\|_{C^{2+\mu,1+\mu/2}(M\times[0,T])}\leq C(\mu,M,N)(\|u_0\|_{C^{2+\nu}(M)}+\kappa_0).
\end{equation*}
By passing to a subsequence, we know that $v_k$ converges to some $u$ in $C^{2,1}(M\times[0,T])$ and $\psi^{v_k+u_0}$ converges to some $\psi$ in $C^0(M\times[0,T])$. Then it is easy to see that $(u,\psi)$ is a solution of \eqref{system in Euclidean1}-\eqref{system in Euclidean2} with $u(\cdot,0)=u_0$ and $\psi(\cdot,0)=\psi_0$.

{\bf Step 2:} $u(x,t)$ takes value in $N$ for any $(x,t)\in M\times[0,T]$.

Suppose $u\in C^{2,1}(M\times[0,T], \mathbb{R}^q)$ and  $\psi\in C^{\mu,\mu/2}(M\times[0,T], \Sigma M\otimes(\pi\circ u)^*TN)\cap L^\infty([0,T];C^{1+\mu}(M))$ satisfy \eqref{system in Euclidean1}-\eqref{system in Euclidean2}. In the following, we write $||\cdot||$ and $\langle\cdot,\cdot\rangle$ for the Euclidean norm and scalar product, respectively. Similarly, we write $||\cdot||_g$ and $\langle\cdot,\cdot\rangle_g$ for the norm and inner product of $(M,g)$, respectively. We define
\begin{equation*}
\rho:\mathbb{R}^q\to\mathbb{R}^q
\end{equation*}
by $\rho(z)=z-\pi(z)$ and
 \begin{equation*}
\varphi: M\times[0,T]\to\mathbb{R}
\end{equation*}
by $\varphi(x,t)=||\rho(u(x,t))||^2=\sum\limits_{A=1}^q|\rho^A(u(x,t))|^2$. A direct computation yields
\begin{equation*}
\begin{split}
(\frac{\partial}{\partial{t}}-\Delta)\varphi(x,t)&=-2\sum\limits_{A=1}^q||\nabla(\rho^A\circ u)(x,t)||_g^2\\
&\quad+2\langle\rho\circ u,-\pi^A_B(u)F^B_1(u)\rangle\\
&\quad+\frac{2}{\alpha(1+|\nabla{u}|^2)^{\alpha-1}}\langle\rho\circ u,\rho^A_B(u)F^B_2(u,\psi)\rangle\\
&\quad+\frac{4(\alpha-1)}{1+|\nabla{u}|^2}\langle\rho\circ u, \nabla^2_{\beta\gamma}u^C\nabla_{\beta}u^C\nabla_{\gamma}u^B\rho_B^A(u)\rangle,
\end{split}\end{equation*}
where $F_1^A$ and $F_2^A$ are defined in \eqref{F1A} and \eqref{F2A}, respectively.

Since $\rho\circ u\in T^\perp_{\pi\circ{u}}N$ and $(d\pi)_{u}: \mathbb{R}^q\to T_{\pi\circ{u}}N$, we have
\begin{equation*}
\langle\rho\circ u,-\pi^A_B(u)F^B_1\rangle=\langle\rho\circ u,\rho^A_B(u)F^B_2\rangle=0.
\end{equation*}
Together with
\begin{equation*}\begin{split}
&\frac{4(\alpha-1)}{1+|\nabla{u}|^2}\langle\rho\circ u, \nabla^2_{\beta\gamma}u^C\nabla_{\beta}u^C\nabla_{\gamma}u^B\rho_B^A(u)\rangle\\
&\leq 4(\alpha-1)||u||_{C^2(M)}||\rho\circ{u}||||\nabla(\rho\circ{u})||\\
&\leq 2(\alpha-1)(||u||^2_{C^2(M)}\varphi+||\nabla(\rho\circ{u})||^2),
\end{split}\end{equation*}
we get
$
(\frac{\partial}{\partial{t}}-\Delta)\varphi(x,t)\leq C\varphi,
$
where $C=C(\|u\|_{C^{2,1}(M\times[0,T])})$. Since $\varphi(x,t)\geq0$ and $\varphi(x,0)=0$ for any $(x,t)\in M\times[0,T]$, we conclude $\varphi=0$ on $M\times[0,T]$. We have shown that $u(x,t)\in N$ for all $(x,t)\in M\times[0,T]$.

Finally, by using the $\epsilon$-regularity (see Lemma \ref{small energy regularity} below), we conclude that
\begin{equation*}
u\in C^{2+\mu,1+\mu/2}(M\times[0,T], N)
\end{equation*} and
\begin{equation*}
\psi\in C^{\mu,\mu/2}(M\times[0,T], \Sigma M\otimes u^*TN)\cap L^\infty([0,T];C^{1+\mu}(M)).
\end{equation*}
\end{proof}

\subsection{Regularity of the flow}

In this subsection, we will give some estimates for the regularity of the flow. The proofs can be found in \cite{jost2019short} and the references therein. Let us start with the following estimate of the energy of the map part.

\begin{lem}\label{map energy}
Suppose $(u,\psi)$ is a solution of \eqref{alpha map}-\eqref{alpha dirac} with the initial values \eqref{initial value}. Then there holds
\begin{equation*}
E_\alpha(u(t))+2\alpha\int_0^t\int_M(1+|\nabla{u}|^2)^{\alpha-1}|\partial_{t}u|^2=E_\alpha(u_0).
\end{equation*}
Moreover, $E_\alpha(u(t))$ is absolutely continuous on $[0,T]$ and non-increasing.
\end{lem}

Consequently, we can also control the spinor part along the heat flow of the $\alpha$-Dirac-harmonic map.
\begin{lem}
Suppose $(u,\psi)$ is a solution of \eqref{alpha map}-\eqref{alpha dirac} with the initial values \eqref{initial value}. Then for any $p\in(1,2)$, there holds
\begin{equation*}
||\psi(\cdot,t)||_{W^{1,p}(M)}\leq C, \ \forall t\in[0,T],
\end{equation*}
where $C=C(p,M,N,E_\alpha(u_0))$.
\end{lem}

To get the convergence of the flow, we also need the following $\epsilon$-regularity.
\begin{lem}\label{small energy regularity}
Suppose $(u,\psi)$ is a solution of \eqref{alpha map}-\eqref{alpha dirac} with the initial values \eqref{initial value}. Given $\omega_0=(x_0,t_0)\in M\times(0,T]$, denote
\begin{equation*}
P_R(\omega_0):=B_R(x_0)\times[t_0-R^2,t_0].
\end{equation*}
Then there exist three constants $\epsilon_2=\epsilon_2(M,N)>0$, $\epsilon_3=\epsilon_3(M,N,u_0)>0$ and $C=C(\mu,R,M,N,E_\alpha(u_0))>0$ such that if
\begin{equation*}
1<\alpha<1+\epsilon_2, \ \text{and} \sup_{[t_0-4R^2,t_0]}E(u(t);B_{2R}(\omega_0))\leq\epsilon_3,
\end{equation*}
then
\begin{equation*}
\sqrt{R}||\psi||_{L^\infty(P_R(\omega_0))}+R||\nabla{u}||_{L^\infty(P_R(\omega_0))}\leq C
\end{equation*}
and for any $0<\beta<1$,
\begin{equation*}
\sup_{[t_0-\frac{R^2}{4},t_0]}||\psi(t)||_{C^{1+\mu}(B_{R/2}(x_0))}+||\nabla{u}||_{C^{\beta,\beta/2}(P_{R/2}(\omega_0))}\leq C(\beta).
\end{equation*}
Moreover, if
\begin{equation*}
\sup_{M}\sup_{[t_0-4R^2,t_0]}E(u(t);B_{2R}(\omega_0))\leq\epsilon_3,
\end{equation*}
then
\begin{equation*}
||u||_{C^{2+\mu,1+\mu/2}(M\times[t_0-\frac{R^2}{8},t_0])}+||\psi||_{C^{\mu,\mu/2}(M\times[t_0-\frac{R^2}{8},t_0])}+\sup_{[t_0-\frac{R^2}{8},t_0]}||\psi(t)||_{C^{1+\mu}(M)}\leq C.
\end{equation*}
\end{lem}

\subsection{Existence of Dirac-harmonic maps}
In this section, we will prove Theorem \ref{flow to DH maps} by the short-time existence of $\alpha$-Dirac-harmonic map flow. First, we will prove the following existence result about the $\alpha$-Dirac-harmonic maps for $\alpha>1$. Then, by the compactness, we get a Dirac-harmonic map as the limit of these $\alpha$-Dirac-harmonic maps. Last, we prove that the bubbles only can be harmonic spheres, and finish the proof of Theorem \ref{flow to DH maps}.

\begin{thm}\label{Existence of alpha DH map}
Let $M$ be a closed spin surface and $(N,h)$ a closed K\"ahler manifold. Suppose there exists a map $u_0\in C^{2+\mu}(M,N)$ for some $\mu\in(0,1)$ such that ${\rm dim}_{\mathbb{H}}{\rm ker}\slashed{D}_{1,0}^{u_0}=1$. Then for any $\alpha\in(1,1+\epsilon_1)$, there exists a nontrivial smooth $\alpha$-Dirac-harmonic map $(u_\alpha,\psi_\alpha)$ such that the map part $u_\alpha$ stays in the same homotopy class as $u_0$ {and $\|\psi_\alpha\|_{L^2}=1$. }
\end{thm}

\begin{proof}[Proof of Theorem \ref{Existence of alpha DH map}]
Let us define
\begin{equation*}
m^\alpha_0:=\inf\{E_\alpha(u)|u\in W^{1,2\alpha}(M,N)\cap[u_0]\},
\end{equation*}
where $[u_0]$ denotes the homotopy class of $u_0$. If $u_0$ is a minimizing $\alpha$-harmonic map, it follows from Lemma \ref{map energy} that $(u_0,\psi_0)$ is an $\alpha$-Dirac-harmonic map for any $\psi_0\in{\rm ker}\slashed{D}_{1,0}^{u_0}$. If $E_\alpha(u_0)>m^\alpha_0$, then Theorem \ref{short-time alpha flow} gives us a solution
\begin{equation*}
u\in C^{2+\mu,1+\mu/2}(M\times[0,T),N)
\end{equation*}
and
\begin{equation*}
\psi\in C^{\mu,\mu/2}(M\times[0,T), \Sigma M\otimes u^*TN)\cap{\cap_{0<s<T}}L^\infty([0,s];C^{1+\mu}(M)).
\end{equation*}
to the problem \eqref{alpha map}-\eqref{alpha dirac} with the initial values \eqref{initial value}.

 By Lemma \ref{map energy}, we know
\begin{equation*}
\int_M(1+|\nabla{u}|^2)^\alpha\leq E_\alpha(u_0).
\end{equation*}
Then it is easy to see that, for any $0<\epsilon<\epsilon_3$, there exists a positive constant $r_0=r_0(\epsilon, \alpha, E_\alpha(u_0))$ such that for all $(x,t)\in M\times[0,T)$, there holds
\begin{equation*}
\int_{B_{r_0}(x)}|\nabla{u}|^2\leq CE_\alpha(u_0)^{1/\alpha}r_0^{1-\frac{1}{\alpha}}\leq \epsilon.
\end{equation*}
Therefore, by Theorem \ref{short-time alpha flow} and Lemma \ref{small energy regularity}, we know that the singular time can be characterized as
\begin{equation*}
Z=\{T\in\mathbb{R}|\lim\limits_{t_i\nearrow T}{\rm dim}_{\mathbb{H}}{\rm ker}\slashed{D}_{1,0}^{u(t_i)}>1\}
\end{equation*}
and there exists a sequence $\{t_i\}\nearrow T$ such that
\begin{equation*}
(u(\cdot,t_i),\psi(\cdot,t_i))\to (u(\cdot,T),\psi(\cdot,T)) \ \text{in} \ C^{2+\mu}(M)\times C^{1+\mu/2}(M)
\end{equation*}
{and
\begin{equation*}
\|\psi(\cdot,T)\|_{L^2}=1.
\end{equation*}
}

If $Z=\emptyset$, then, by Theorem \ref{short-time alpha flow}, we can extend the solution $(u,\psi)$ beyond the time $T$ by using $(u(\cdot,T),\psi(\cdot,T))$ as new initial values. Thus, we have the global existence of the flow. For the limit behavior as $t\to\infty$,  Lemma \ref{map energy} implies that there exists a sequence $\{t_i\}\to\infty$ such that
\begin{equation}\label{time derivative}
\int_M|\partial_tu|^2(\cdot,t_i)\to 0.
\end{equation}
Together with  Lemma \ref{small energy regularity}, there is a subsequence, still denoted by $\{t_i\}$, and an $\alpha$-Dirac-harmonic map $(u_\alpha, \psi_\alpha)\in C^\infty(M,N)\times C^\infty(M,\Sigma M\otimes(u_\alpha)^*TN)$ such that $(u(\cdot,t_i),\psi(\cdot,t_i))$ converges to $(u_\alpha, \psi_\alpha)$ in $C^2(M)\times C^1(M)$ {and $\|\psi_\alpha\|_{L^2}=1$.}

If $Z\neq\emptyset$ and $T\in{Z}$, let us assume that $E_\alpha(u(\cdot,T))>m^\alpha_0$ and $(u(\cdot,T),\psi(\cdot,T))$ is not already an $\alpha$-Dirac-harmonic map.  We extend the flow as follows: By Lemma \ref{density}, there is a map $u_1\in C^{2+\mu}(M,N)$ such that
 \begin{equation}\label{lower energy}
 m^\alpha_0<E_\alpha(u_1)<E_\alpha(u(\cdot,T))
 \end{equation}
 and
 \begin{equation}\label{minimal kernel}
 {\rm dim}_{\mathbb{H}}{\rm ker}\slashed{D}_{1,0}^{u_1}=1.
 \end{equation}
Thus, picking any $\psi_1\in{\rm ker}\slashed{D}^{u_1}$ with $\|\psi_1\|_{L^2}=1$, we can restart the flow from the new initial values $(u_1,\psi_1)$. If there is no singular time along the flow started from $(u_1,\psi_1)$, then we get an $\alpha$-Dirac-harmonic map as in the case of $Z=\emptyset$. Otherwise, we use again the procedure above to choose $(u_2,\psi_2)$ as  initial values and restart the flow.  This procedure will stop in  finitely or infinitely many steps.

 If infinitely many steps are required, then there exist infinitely many  flow pieces  $\{u_i(x,t)\}_{i=1,\dots,\infty}$ and $\{T_i\}_{i=1,\dots,\infty}$ such that
\begin{equation*}
E_\alpha(u_i(t))+2\alpha\int_0^t\int_M(1+|\nabla{u}_i|^2)^{\alpha-1}|\partial_{t}u_i|^2=E_\alpha(u_i), \ \forall t\in(0,T_i),
\end{equation*}
where $u_i(\cdot,0)=u_i\in C^{2+\mu}(M,N)$.  If the $T_i$ are bounded away from zero, then there is $\{t_i\}$ such that \eqref{time derivative} holds for $t_i\in(0,T_i)$. Therefore, we have an $\alpha$-Dirac-harmonic map as before. If $T_i\to 0$, then we look at the limit of $E_\alpha(u_i)$. If the limit is strictly bigger than $m^\alpha_0$, we again choose another map satisfying \eqref{lower energy} and \eqref{minimal kernel} as a new starting point. If the limit is exactly $m^\alpha_0$, then we choose $\{t_i\}$ such that $t_i\in(0,T_i)$ for each $i$. By Lemma \ref{small energy regularity},  $u_i(t_i)$ converges in $C^2(M)\times C^1(M)$ to a minimizing $\alpha$-harmonic map $u_\alpha$. If $\slashed{D}_{1,0}^{u_\alpha}$ has minimal kernel, then for any $\psi\in{\rm ker}{\slashed{D}_{1,0}^{u_\alpha}}$, $(u_\alpha,\psi)$ is an $\alpha$-Dirac-harmonic map as we showed in the beginning of the proof. If $\slashed{D}_{1,0}^{u_\alpha}$ has non-minimal kernel, by using the $\mathbb{Z}_2$-grading $G\otimes id$ as in the proof of Theorem \ref{inv}, we get  $\alpha$-Dirac-harmonic maps $(u_\alpha,\psi_\alpha^\pm)$ for any ${\rm ker}{\slashed{D}^{u_\alpha}}\ni\psi_\alpha=\psi_\alpha^++\psi_\alpha^-$. In particular, we can choose $\psi_\alpha$ such that $\|\psi_\alpha^+\|_{L^2}=1$ or $\|\psi_\alpha^-\|_{L^2}=1$. By this procedure, we either get an $\alpha$-Dirac-harmonic map or keep on  choosing new maps satisfying \eqref{lower energy} and \eqref{minimal kernel}.  In the latter case, since the energies of the initial maps are bounded and decreasing, they converge to the minimizing energy $m^\alpha_0$. (Otherwise, suppose the constant is $C>m^\alpha_0$. Then one can choose a new map with a lower energy such that the limit is not $C$.) Therefore, we also get an $\alpha$-Dirac-harmonic map in the latter case as before.

If it stops in finitely many steps, there exists a sequence $\{t_i\}$ and some $0<T_k\leq+\infty$ such that
\begin{equation*}
\lim\limits_{t_i\nearrow T}(u(\cdot,t_i),\psi(\cdot,t_i))\to (u_\alpha,\psi_\alpha) \ \text{in} \ C^2(M)\times C^1(M),
\end{equation*}
where $(u_\alpha,\psi_\alpha)$ either is an $\alpha$-Dirac-harmonic map or satisfies $E_\alpha(u_\alpha)=m^\alpha_0$. And in the latter case, $u_\alpha$ is a minimizing $\alpha$-harmonic map. Then we can again get a nontrivial $\alpha$-Dirac-harmonic map as above.

\end{proof}

By Theorem \ref{Existence of alpha DH map}, for any $\alpha>1$ sufficiently close to $1$, there exists an $\alpha$-Dirac-harmonic map $(u_\alpha,\psi_\alpha)$ with the properties
\begin{equation}\label{map energy bound}
E_\alpha(u_\alpha)\leq E_\alpha(u_0), \  \ \|\psi_\alpha\|_{L^2}=1
\end{equation}
and
\begin{equation}\label{Lp of spinor}
||\psi_\alpha||_{W^{1,p}(M)}\leq C(p,M,N,E_\alpha(u_0))
\end{equation}
for any $1<p<2$. Then it is natural to consider the limit behavior when $\alpha$ decreases to $1$. Together with the blow-up analysis in \cite{jost2018geometric}, we have the following existence result.

\begin{thm}\label{flow to DH maps}
Let $M$ be a closed Riemann surface and $N$ a complex $n$-dimensional analytic K\"ahler manifold  and a parallel real structure be $j_2$ defined in Theorem \ref{inv}. Suppose there exists a map $u_0\in C^{2+\mu}(M,N)$ for some $\mu\in(0,1)$ such that ${\rm dim}_{\mathbb{H}}{\rm ker}\slashed{D}_{1,0}^{u_0}=1$. Then there exists a nontrivial (i.e. $\Psi\neq0$) smooth Dirac-harmonic map $(\Phi,\Psi)$ with
$\|\Psi\|_{L^2}=1$.
 In particular, if $N$ has nonpositive curvature, then the map $\Phi$ stays in the same homotopy class as $u_0$. 
 \end{thm}

\begin{proof}
By Theorem \ref{Existence of alpha DH map}, we have a sequence of smooth $\alpha$-Dirac-harmonic maps $(u_{\alpha_k},\psi_{\alpha_k})$ with \eqref{map energy bound} and \eqref{Lp of spinor}, where $\alpha_k\searrow1$ as $k\to\infty$. Then, by the compactness theorem in \cite{jost2018geometric}, there is a constant $\epsilon_0>0$ and a Dirac-harmonic map
$$
(\Phi,\Psi)\in C^\infty(M,N)\times C^\infty(M,\Sigma M\otimes\Phi^*TN)
$$
such that
\begin{equation*}
(u_{\alpha_k},\psi_{\alpha_k})\to(\Phi,\Psi) \ \text{in} \  C^2_{loc}(M\setminus{\mathcal S})\times C^1_{loc}(M\setminus{\mathcal S}),
\end{equation*}
where
\begin{equation*}
\mathcal{S}:=\{x\in M|\liminf_{\alpha_k\to1}E(u_{\alpha_k};B_{r}(x))\geq\frac{\epsilon_0}{2}, \forall r>0 \}
\end{equation*}
is a finite set.

Now, taking $x_0\in\mathcal{S}$, there exists a sequence $x_{\alpha_k}\to x_0$, $\lambda_{\alpha_k}\to0$ and a nontrivial Dirac-harmonic map $(\phi,\xi): \mathbb{R}^2\to N$ such that
\begin{equation*}
(u_{\alpha_k}(x_{\alpha_k}+\lambda_{\alpha_k}x),\lambda_{\alpha_k}^{{\alpha_k}-1}\sqrt{\lambda_{\alpha_k}}\psi_{\alpha_k}(x_{\alpha_k}+\lambda_{\alpha_k}x))\to(\phi,\xi) \ \text{in} \ C^2_{loc}(\mathbb{R}^2),
\end{equation*}
as $\alpha\to1$. Choose any $p^*>4$, by taking $p=\frac{2p^*}{2+p^*}$ in \eqref{Lp of spinor}, we get
\begin{equation*}
||\psi_{\alpha_k}||_{L^{p^*}(M)}\leq C(p^*,M,N,E^{\alpha_k}(u_0))
\end{equation*}
and
\begin{equation*}
\begin{split}
||\xi||_{L^4(D_R(0))}&=\lim\limits_{{\alpha_k}\to1}\lambda^{{\alpha_k}-1}_{\alpha_k}||\psi_{\alpha_k}||_{L^4(D_{\lambda_{\alpha_k}{R}}(x_{\alpha_k}))}\\
&\leq\lim\limits_{{\alpha_k}\to1}C||\psi_{\alpha_k}||_{L^{p^*}(M)}(\lambda_{\alpha_k}{R})^{2(\frac14-\frac{1}{p^*})}=0.
\end{split}
\end{equation*}
Thus, $\xi=0$ and $\phi$ can be extended to a nontrivial smooth harmonic sphere. Since $||\psi_\alpha||_{L^2}=1$, the Sobolev embedding implies that $||\Psi||_{L^2(M)}=\lim\limits_{{\alpha_k}\to1}||\psi_\alpha||_{L^2(M)}=1$. Therefore, $(\Phi,\Psi)$ is nontrivial. Furthermore, if $(N,h)$ does not admit any nontrivial harmonic sphere, then
\begin{equation*}
(u_{\alpha_k},\psi_{\alpha_k})\to(\Phi,\Psi) \ \text{in} \  C^2(M)\times C^1(M).
\end{equation*}
Therefore, $\Phi$ is in the same homotopy class as $u_0$.
\end{proof}

\vspace{4em}
\section{Appendix}
We will use the parallel construction in \cite{jost2019short} to construct the solution to the constraint equation for spinors under a different pull-back bundle $u^*T_{1,0}N$. Since the only thing changed is the bundle we twisted, the proofs of those nice properties are parallel to those in \cite{jost2019short}. For completeness, we give the details in this appendix.

For every $T>0$, we consider the space $B^T_R(\bar{u}_0):=\{u\in X_T|\|u-\bar{u}_0\|_{X_T}\leq R\}\cap\{u|_{t=0}=u_0\}$ where $\bar{u}_0(x,t)=u_0(x)$ for any $t\in[0,T]$. To get the necessary estimate for the solution of the constraint equation, we will use the parallel transport along the unique shortest geodesic between $u_0(x)$ and $\pi\circ{u_t}(x)$ in N. To do this, we need the following lemma which tells us that the distances in $N$ can be locally controlled by the distances in $\mathbb{R}^q$.

\begin{lem}\label{distance control}
Let $N\subset\mathbb{R}^q$ be a closed embedded submanifold of $\mathbb{R}^q$ with the induced Riemannian metric. Denote by $A$ its Weingarten map. Choose $C>0$ such that $||A||\leq C$, where
\begin{equation*}
||A||:=\sup\{||A_vX||| \ v\in T_p^\perp{N}, \ X\in T_pN, \ ||v||=1, \ ||X||=1, \ p\in N\}.
\end{equation*}
Then there exists $0<\delta_0<\frac1C$ such that for all $0<\delta\leq\delta_0$ and for all $p,q\in N$ with $||p-q||_2<\delta$, it holds that
\begin{equation*}
d^N(p,q)\leq\frac{1}{1-\delta{C}}||p-q||_2,
\end{equation*}
where we denote the Euclidean norm by $||\cdot||_2$ in this section.
\end{lem}

In the following, we will choose $\delta$ and $R$ to ensure the existence of the unique shortest geodesics between the projections of any two elements in $B^T_R(\bar{u}_0)$. By the definition of $B^T_R(\bar{u}_0)$, we have
\begin{equation*}
||u(x,t)-\bar{u}_0(x,t)||_2=||u(x,t)-u_0(x)||_2\leq R
\end{equation*}
for all $(x,t)\in M\times[0,T]$. Then taking any $R\leq\delta$, we get
\begin{equation*}
d(u(x,t),N)\leq||u(x,t)-u_0(x)||_2\leq\delta
\end{equation*}
for all $(x,t)\in M\times[0,T]$. Therefore, $u(x,t)\in N_\delta$. In particular, $\pi\circ{u}$ is $N$-valued, and
\begin{equation}\label{proj u}
||(\pi\circ{u})(x,t)-u_0(x)||_2\leq||(\pi\circ{u})(x,t)-u(x,t)||_2+||u(x,t)-u_0(x)||_2\leq2\delta.
\end{equation}
Now, we choose $\epsilon>0$ with $2\epsilon<{\rm inj}(N)$ and $\delta$ such that
\begin{align}\label{delta}
\delta<\min\set{\frac14\delta_0,\frac14\epsilon (1-\delta_0C)},
\end{align}
where $\delta_0,C>0$ are as in Lemma \ref{distance control}. From \eqref{proj u}, we know that for all $u,v\in B^T_R(\bar{u}_0)$, it holds that
\begin{equation*}
||(\pi\circ{u})(x,t)-(\pi\circ{v})(x,s)||_2\leq4\delta<\delta_0.
\end{equation*}
Then Lemma \ref{distance control} and \eqref{delta} imply that
\begin{equation}\label{control distance on N}
\begin{split}
d^N((\pi\circ{u})(x,t),(\pi\circ{v})(x,s))&\leq\frac{1}{1-\delta_0C}||(\pi\circ{u})(x,t)-(\pi\circ{v})(x,s)||_2\\
&\leq\frac{1}{1-\delta_0C}4\delta<\epsilon<\frac12{\rm inj}(N).
\end{split}\end{equation}
To summarize, under the choice of constants as follows:
\begin{equation}\label{choice of constants}
		\begin{cases}
		\epsilon>0, & \text{s.t.} \ 2\epsilon<{\rm inj}(N), \\
		\delta>0, & \text{s.t.} \ \delta<\min\{\frac14\delta_0,\frac14\epsilon(1-\delta_0C)\}, \\
		R\leq\delta,
				\end{cases}
	\end{equation}
we have shown that
\begin{equation}\label{small nbhd}
u(x,t)\in N_\delta
\end{equation}
and
\begin{equation}\label{shortest geodesic}
d^N((\pi\circ u)(x,t),(\pi\circ v)(x,s))<\epsilon<\frac12{\rm inj}(N)
\end{equation}
for all $u,v\in B_R^T(\bar{u}_0)$, $x\in M$ and $t,s\in[0,T]$.

Using the  properties \eqref{small nbhd} and \eqref{shortest geodesic}, we can prove two important estimates. One is for the Dirac operators along maps.

\begin{lem}\label{Dirac along maps u}
Choose $\epsilon$, $\delta$ and $R$ as in \eqref{choice of constants}. If $\epsilon>0$ is small enough, then there exists $C=C(R)>0$ such that
\begin{equation*}
||((P^{v_s,u_t})^{-1}\slashed{D}_{1,0}^{\pi\circ{u_t}}P^{v_s,u_t}-\slashed{D}_{1,0}^{\pi\circ{v_s}})\psi(x)||\leq C||u_t-v_s||_{C^0(M,\mathbb{R}^q)}||\psi(x)||
\end{equation*}
for any $u,v\in B^T_R(\bar{u}_0)$, $\psi\in\Gamma_{C^1}(\Sigma M\otimes(\pi\circ{v_s})^*T_{1,0}N)$, $x\in M$ and $t,s\in[0,T]$.
\end{lem}

\begin{proof}
We write $f_0:=\pi\circ{v_s}$, $f_1:=\pi\circ{u_t}$ and define the $C^1$ map $F: M\times[0,1]\to N$ by
\begin{equation*}
F(x,t):=\exp_{f_0(x)}(t\exp^{-1}_{f_0(x)}f_1(x))
\end{equation*}
where $\exp$ denotes the exponential map of the Riemannian manifold $N$. Note that $F(\cdot,0)=f_0$, $F(\cdot,1)=f_1$ and $t\mapsto F(x,t)$ is the unique shortest geodesic from $f_0(x)$ to $f_1(x)$. We denote by
\begin{equation*}
\mathcal{P}_{t_1,t_2}=\mathcal{P}_{t_1,t_2}(x): T_{1,0}N|_{F(x,t_1)}\to T_{1,0}N|_{F(x,t_2)}
\end{equation*}
the parallel transport in $F^*T_{1,0}N$ with respect to $\nabla^{F^*T_{1,0}N}$ (pullback of the connection on $T_{1,0}N$) along the curve $\gamma_x(t):=(x,t)$ from $\gamma_x(t_1)$ to $\gamma_x(t_2)$, $x\in M$, $t_1,t_2\in[0,1]$. In particular, $\mathcal{P}_{0,1}=P^{v_s,u_t}$. Let $\psi\in\Gamma_{C^1}(\Sigma M\otimes(f_0)^*T_{1,0}N)$. We have
\begin{equation}\label{difference of Dirac operaor}
\begin{split}
&((\mathcal{P}_{0,1})^{-1}\slashed{D}^{f_1}\mathcal{P}_{0,1}-\slashed{D}^{f_0})\psi\\
&=(e_\alpha\cdot\psi^i)\otimes(((\mathcal{P}_{0,1})^{-1}\nabla_{e_\alpha}^{f_1^*T_{1,0}N}\mathcal{P}_{0,1}-\nabla_{e_\alpha}^{f_0^*T_{1,0}N})(b_i\circ{f_0}))
\end{split}
\end{equation}
where $\psi=\psi^i\otimes(b_i\circ{f_0})$, $\{b_i\}$ is an orthonormal frame of $T_{1,0}N$, $\psi^i$ are local $C^1$ sections of $\Sigma M$, and $\{e_\alpha\}$ is an orthonormal frame of $TM$.

We define local $C^1$ sections $\Theta_i$ of $F^*T_{1,0}N$ by
\begin{equation*}
\Theta_i(x,t):=\mathcal{P}_{0,t}(x)(b_i\circ{f_0})(x).
\end{equation*}
For each $t\in[0,1]$ we define the functions $T_{ij}(\cdot,t):=T_{ij}^\alpha(\cdot,t)$ by
\begin{equation}\label{Tij}
(\mathcal{P}_{0,t})^{-1}((\nabla_{e_\alpha}^{F^*T_{1,0}N}\Theta_i)(x,t))=\sum_jT_{ij}^\alpha(x,t)(b_j\circ{f_0})(x).
\end{equation}
So far, we only know that the $T_{ij}$ are continuous. In the following, we will perform some formal calculations and justify them afterwards. By a straightforward computation, we have
\begin{equation}\label{difference of T}
\begin{split}
&||((\mathcal{P}_{0,1})^{-1}\nabla_{e_\alpha}^{f_1^*T_{1,0}N}\mathcal{P}_{0,1}-\nabla_{e_\alpha}^{f_0^*T_{1,0}N})(b_i\circ{f_0})(x)||^2\\
&=||(\mathcal{P}_{0,1})^{-1}((\nabla_{e_\alpha}^{F^*T_{1,0}N}\Theta_i)(x,1))-(\mathcal{P}_{0,0})^{-1}((\nabla_{e_\alpha}^{F^*T_{1,0}N}\Theta_i)(x,0))||^2\\
&=||\sum_jT_{ij}(x,1)(b_j\circ{f_0})(x)-\sum_jT_{ij}(x,0)(b_j\circ{f_0})(x)||^2\\
&=\sum_j(T_{ij}(x,1)-T_{ij}(x,0))^2\\
&=\sum_j\left(\int_0^1\frac{d}{dt}\bigg|_{t=r}T_{ij}(x,t)dr\right)^2.
\end{split}\end{equation}
Therefore we want to control the first time-derivative of the $T_{ij}$. Equation \eqref{Tij} implies that these time-derivatives are related to the curvature of $F^*T_{1,0}N$. More precisely, for all $X\in\Gamma(TM)$ we have
\begin{equation}\label{derivative of T}
\begin{split}
&\frac{d}{dt}\bigg|_{t=r}\left((\mathcal{P}_{0,t})^{-1}\left((\nabla_{X}^{F^*T_{1,0}N}\Theta_i)(x,t)\right)\right)\\
&=\frac{d}{dt}\bigg|_{t=0}\left((\mathcal{P}_{0,t+r})^{-1}\left((\nabla_{X}^{F^*T_{1,0}N}\Theta_i)(x,t+r)\right)\right)\\
&=\frac{d}{dt}\bigg|_{t=0}\left((\mathcal{P}_{0,r})^{-1}(\mathcal{P}_{r,r+t})^{-1}\left((\nabla_{X}^{F^*T_{1,0}N}\Theta_i)(x,t+r)\right)\right)\\
&=(\mathcal{P}_{0,r})^{-1}\frac{d}{dt}\bigg|_{t=0}\left((\mathcal{P}_{r,r+t})^{-1}\left((\nabla_{X}^{F^*T_{1,0}N}\Theta_i)(x,t+r)\right)\right)\\
&=(\mathcal{P}_{0,r})^{-1}\left((\nabla_{\frac{\partial}{\partial t}}^{F^*T_{1,0}N}\nabla_{X}^{F^*T_{1,0}N}\Theta_i)(x,r)\right).
\end{split}\end{equation}
Now, let us justify the formal calculations \eqref{difference of T} and \eqref{derivative of T}. Combining the definition of $\Theta_i$ as parallel transport and a careful examination of the regularity of F we
deduce that $(\nabla_{\frac{\partial}{\partial t}}^{F^*T_{1,0}N}\nabla_{X}^{F^*T_{1,0}N}\Theta_i)(x,r)$ exists. Then \eqref{derivative of T} holds. Together with \eqref{Tij}, we know that the $T_{ij}$ are differentiable in $t$. Therefore \eqref{difference of T} also holds. We further get
\begin{equation*}
\begin{split}
&\nabla_{\frac{\partial}{\partial t}}^{F^*T_{1,0}N}\nabla_{X}^{F^*T_{1,0}N}\Theta_i\\
&=R^{F^*T_{1,0}N}(\frac{\partial}{\partial t},X)\Theta_i+\nabla_{X}^{F^*T_{1,0}N}\nabla_{\frac{\partial}{\partial t}}^{F^*T_{1,0}N}\Theta_i-\nabla_{[\frac{\partial}{\partial t},X]}^{F^*T_{1,0}N}\Theta_i\\
&=R^{F^*T_{1,0}N}(\frac{\partial}{\partial t},X)\Theta_i=R^{T_{1,0}N}(dF(\frac{\partial}{\partial t}),dF(X))\Theta_i,
\end{split}\end{equation*}
since $\nabla_{\frac{\partial}{\partial t}}^{F^*T_{1,0}N}\Theta_i=0$ by the definition of $\Theta_i$ and $[\frac{\partial}{\partial t},X]=0$.

This implies
\begin{equation*}
\begin{split}
\sum_j\left(\frac{d}{dt}\bigg|_{t=r}T_{ij}(x,t)\right)^2
&=||\frac{d}{dt}\bigg|_{t=r}\left((\mathcal{P}_{0,t})^{-1}((\nabla_{e_\alpha}^{F^*T_{1,0}N}\Theta_i)(x,t))\right)||^2\\
&=||\left(\nabla_{\frac{\partial}{\partial t}}^{F^*T_{1,0}N}\nabla_{e_\alpha}^{F^*T_{1,0}N}\Theta_i\right)(x,r)||^2\\
&=||R^{T_{1,0}N}(dF_{(x,r)}(\frac{\partial}{\partial t}),dF_{(x,r)}(e_\alpha))\Theta_i(x,r)||^2\\
&\leq C_1||dF_{(x,r)}({\partial_t})||^2||dF_{(x,r)}(e_\alpha))||^2,
\end{split}\end{equation*}
where $C_1$ only depends on $N$.

In the following we estimate $||dF_{(x,r)}({\partial_t})||$ and $||dF_{(x,r)}(e_\alpha))||$. We have
\begin{equation*}
dF_{(x,r)}({\partial_t}|_{(x,r)})=\frac{\partial}{\partial t}\bigg|_{t=r}(\exp_{f_0(x)}(t\exp^{-1}_{f_0(x)}f_1(x)))=c'(r),
\end{equation*}
where $c(t):=\exp_{f_0(x)}(t\exp^{-1}_{f_0(x)}f_1(x))$ is a geodesic in $N$. In particular, $c'$ is parallel along $c$ and thus $||c'(r)||=||c'(0)||=||\exp^{-1}_{f_0(x)}f_1(x)||$. Therefore, we get
\begin{equation*}
||dF_{(x,r)}({\partial_t})||=||\exp^{-1}_{f_0(x)}f_1(x)||\leq d^N(f_0(x),f_1(x))\leq C_2||u_t-v_s||_{C^0(M,\mathbb{R}^q)},
\end{equation*}
where we have used Lemma \ref{distance control} and the Lipschitz continuity of $\pi$. Moreover, there exists $C_3(R)>0$ such that $||dF_{(x,r)}(e_\alpha))||\leq C_3(R)$ for all $(x,r)\in M\times[0,1]$. Thus,  we have shown
\begin{equation*}
\sum_j\left(\frac{d}{dt}\bigg|_{t=r}T_{ij}(x,t)\right)^2\leq C_1C_2^2C_3(R)^2||u_t-v_s||_{C^0(M,\mathbb{R}^q)}^2
\end{equation*}
for all $(x,t)$. Combining this with \eqref{difference of Dirac operaor} and \eqref{difference of T}, we complete the proof.
\end{proof}

The other one is for the parallel transport.
\begin{lem}\label{PT}
Choose $\epsilon$, $\delta$ and $R$ as in \eqref{choice of constants}. If $\epsilon>0$ is small enough, then there exists $C=C(\epsilon)>0$ such that
\begin{equation*}
||P^{v_s,u_0}P^{u_t,v_s}P^{u_0,u_t}Z-Z||\leq C||u_t-v_s||_{C^0(M,\mathbb{R}^q)}||Z||
\end{equation*}
for all $Z\in T_{1,0}N|_{u_0(x)}$, $u,v\in B^T_R(\bar{u}_0)$, $x\in M$ and $t,s\in[0,T]$.
\end{lem}

Consequently, we also have
\begin{lem}\label{W1p norm}
Choose $\epsilon$, $\delta$ and $R$ as in \eqref{choice of constants}. For $u,v\in B^T_R(\bar{u}_0)$, $s,t\in[0,T]$, the operator norm of the isomorphism of Banach spaces
\begin{equation*}
P^{v_s,u_t}: \Gamma_{W^{1,p}}(\Sigma M\otimes(\pi\circ{v_s})^*T_{1,0}N)\to\Gamma_{W^{1,p}}(\Sigma M\otimes(\pi\circ{u_t})^*T_{1,0}N)
\end{equation*}
is uniformly bounded, i.e. there exists $C=C(R,p)$ such that
\begin{equation*}
||P^{v_s,u_t}||_{L(W^{1,p},W^{1,p})}\leq C
\end{equation*}
for all $u,v\in B^T_R(\bar{u}_0)$, $x\in M$ and $t,s\in[0,T]$.
\end{lem}

The proofs of these two lemmas only depend on the existence of the unique shortest geodesic between any two maps in $B^T_R(\bar{u}_0)$, which was already shown in \eqref{shortest geodesic}. So we omit them here. Besides, by Lemma \ref{Dirac along maps u}, one can immediately prove the following Lemma by the Min-Max principle.

\begin{lem}\label{dim of kernel u}
Assume that ${\rm dim}_{\mathbb{H}}{\rm ker}(\slashed{D}_{1,0}^{u_0})=1$. Choose $\epsilon$, $\delta$ and $R$ as in Lemma \ref{Dirac along maps u}. If $R$ is small enough, then
\begin{equation*}
{\rm dim}_{\mathbb{H}}{\rm ker}(\slashed{D}_{1,0}^{\pi\circ u_t})=1
\end{equation*}
and there exists $\Lambda=\frac12\Lambda(u_0)$ such that
\begin{equation*}
\#\{{\rm spec}(\slashed{D}_{1,0}^{\pi\circ u_t})\cap[-\Lambda,\Lambda]\}=1
\end{equation*}
for any $u\in B^T_R(\bar{u}_0)$ and $t\in[0,T]$, where $\Lambda(u_0)$ is a constant such that ${\rm spec}(\slashed{D}_{1,0}^{u_0})\setminus\{0\}\subset\mathbb{R}\setminus(-\Lambda(u_0),\Lambda(u_0))$.
\end{lem}

Once we have the minimality of the kernel in Lemma \ref{dim of kernel u}, we can prove the following uniform bounds for the resolvents, which are  important for the Lipschitz continuity of the solution to the Dirac equation.
\begin{lem}\label{resolvents}
Assume we are in the situation of Lemma \ref{dim of kernel u}. We consider the resolvent $R(\lambda, \slashed{D}_{1,0}^{\pi\circ{u_t}}): \Gamma_{L^2}\to\Gamma_{L^2}$ of $\slashed{D}_{1,0}^{\pi\circ{u_t}}:\Gamma_{W^{1,2}}\to\Gamma_{L^2}$. By the $L^p$ estimate (see Lemma 3.3 in \cite{chen2017estimates}), we know the restriction
\begin{equation*}
R(\lambda, \slashed{D}_{1,0}^{\pi\circ{u_t}}): \Gamma_{L^p}\to\Gamma_{W^{1,p}}
\end{equation*}
is well-defined and bounded for any $2\leq p<\infty$. If $R>0$ is small enough, then there exists $C=C(p,R)>0$ such that
\begin{equation*}
\sup_{|\lambda|=\frac{\Lambda}{2}}||R(\lambda, \slashed{D}_{1,0}^{\pi\circ{u_t}})||_{L(L^p,W^{1,p})}<C
\end{equation*}
for any $u\in B^T_R(\bar{u}_0)$, $t\in[0,T]$.
\end{lem}

Now, by the projector of the Dirac operator, we can construct a solution to the constraint equation whose nontriviality follows from the following lemma.
\begin{lem}\label{projection to kernel}
In the situation of Lemma \ref{dim of kernel u}, for any fixed $u\in B^T_R(\bar{u}_0)$ and any $\psi_0\in{\rm ker}(\slashed{D}^{u_0})$ with $\|\psi_0\|_{L^2}=1$, we have
\begin{equation*}
\sqrt{\frac12}\leq\|\tilde\psi_1^{u_t}\|_{L^2}\leq1,
\end{equation*}
where $\tilde\psi^{u_t}=P^{u_0,u_t}\psi_0=\tilde\psi_1^{u_t}+\tilde\psi_2^{u_t}$ with respect to the decomposition $\Gamma_{L^2}={\rm ker}(\slashed{D}_{1,0}^{\pi\circ u_t})\oplus({\rm ker}(\slashed{D}_{1,0}^{\pi\circ u_t}))^\bot$
\end{lem}

In Section 3, to show the short-time existence of the heat flow for $\alpha$-Dirac-harmonic maps, we need the following Lipschitz estimate.

\begin{lem}\label{lip}
Choose $\delta$ as in \eqref{choice of constants}, $\epsilon$ as in Lemma \ref{Dirac along maps u} and Lemma \ref{PT}, $R$ as in Lemma \ref{dim of kernel u} and Lemma \ref{resolvents}. For any harmonic spinor $\psi_0\in{\rm ker}(\slashed{D}_{1,0}^{u_0})$, we define
\begin{equation*}
\bar\psi(u_t):=\tilde\psi^{u_t}_1=-\frac{1}{2\pi i}\int_{\gamma}R(\lambda,\slashed{D}_{1,0}^{\pi\circ u_t})P^{u_0,u_t}\psi_0d\lambda
\end{equation*}
for any $u\in B^T_R(\bar{u}_0)$, where $\gamma$ is defined in the Section $2$ with $\Lambda=\frac12\Lambda(u_0)$. In particular, $\bar\psi(u_t)\in{\rm ker}(\slashed{D}_{1,0}^{\pi\circ u_t})\subset\Gamma_{C^0}(\Sigma M\otimes(\pi\circ u_t)^*T_{1,0}N)$. We write
\begin{equation*}
\psi(u_t):=\psi(u(\cdot,t))=\frac{\bar\psi(u_t)}{\|\bar\psi(u_t)\|_{L^2}}.
\end{equation*}
Let $\psi^A(u_t)$ be the sections of $\Sigma M$ such that
\begin{equation*}
\psi(u_t)=\psi^A(u_t)\otimes(\partial_A\circ\pi\circ u_t)
\end{equation*}
for $A=1,\cdots,q$. Then there exists $C=C(R,\epsilon,\psi_0)>0$ such that
\begin{equation}\label{difference of psi bar after pt}
\|P^{u_t,v_s}\bar\psi(u_t)(x)-\bar\psi(u_t)(x)\|\leq C\|u_t-v_s\|_{C^0(M,\mathbb{R}^q)}
\end{equation}
and
\begin{equation}\label{psi lip}
\|\psi^A(u_t)(x)-\psi^A(v_s)(x)\|\leq C\|u_t-v_s\|_{C^0(M,\mathbb{R}^q)}
\end{equation}
for all $u,v\in B^T_R(\bar{u}_0)$, $A=1,\cdots,q$, $x\in M$ and $s,t\in[0,T]$.
\end{lem}

\begin{proof}
Using the following resolvent identity for two operators $D_1,D_2$
\begin{equation*}
R(\lambda,D_1)-R(\lambda,D_2)=R(\lambda,D_1)\circ(D_1-D_2)\circ R(\lambda,D_2),
\end{equation*}
we have
\begin{equation*}
\begin{split}
&P^{u_t,v_s}\bar\psi(u_t)-\bar\psi(v_s)\\
%&=-\frac{1}{2\pi i}\bigg(\int_{\gamma}R(\lambda,P^{u_t,v_s}\slashed{D}_{1,0}^{\pi\circ u_t}(P^{u_t,v_s})^{-1})P^{u_t,v_s}P^{u_0,u_t}\psi_0\\
%&\quad-\int_{\gamma}R(\lambda,\slashed{D}_{1,0}^{\pi\circ v_s})P^{u_0,v_s}\psi_0\bigg)\\
&=-\frac{1}{2\pi i}\int_{\gamma}R(\lambda,P^{u_t,v_s}\slashed{D}_{1,0}^{\pi\circ u_t}(P^{u_t,v_s})^{-1})\bigg(P^{u_t,v_s}P^{u_0,u_t}\psi_0-P^{u_0,v_s}\psi_0\bigg)\\
&\quad-\frac{1}{2\pi i}\int_{\gamma}\bigg(R(\lambda,P^{u_t,v_s}\slashed{D}_{1,0}^{\pi\circ u_t}(P^{u_t,v_s})^{-1})-R(\lambda,\slashed{D}_{1,0}^{\pi\circ v_s})\bigg)P^{u_0,v_s}\psi_0\\
&=-\frac{1}{2\pi i}\int_{\gamma}R(\lambda,P^{u_t,v_s}\slashed{D}_{1,0}^{\pi\circ u_t}(P^{u_t,v_s})^{-1})\bigg(P^{u_t,v_s}P^{u_0,u_t}\psi_0-P^{u_0,v_s}\psi_0\bigg)\\
&\begin{split}
\quad-\frac{1}{2\pi i}\int_{\gamma}&\bigg(R(\lambda,P^{u_t,v_s}\slashed{D}_{1,0}^{\pi\circ u_t}(P^{u_t,v_s})^{-1})\circ\left(P^{u_t,v_s}\slashed{D}_{1,0}^{\pi\circ u_t}(P^{u_t,v_s})^{-1}-\slashed{D}_{1,0}^{\pi\circ v_s}\right)\circ\\
&\quad R(\lambda,\slashed{D}_{1,0}^{\pi\circ v_s})\bigg)P^{u_0,v_s}\psi_0,
\end{split}
\end{split}\end{equation*}
where $\gamma$ is defined in \eqref{general curve} with $\Lambda=\frac12\Lambda(u_0)$. Therefore, for $p$ large enough, we get
\begin{equation*}
\begin{split}
&||P^{u_t,v_s}\bar\psi(u_t)(x)-\bar\psi(v_s)(x)||\leq C_1||P^{u_t,v_s}\bar\psi^{u_t}-\bar\psi^{v_s}||_{W^{1,p}(M)}\\
&\leq C_2\bigg\|\int_{\gamma}R(\lambda,P^{u_t,v_s}\slashed{D}^{\pi\circ u_t}(P^{u_t,v_s})^{-1})\bigg(P^{u_t,v_s}P^{u_0,u_t}\psi_0-P^{u_0,v_s}\psi_0\bigg)\bigg\|_{W^{1,p}(M)}\\
&\begin{split}
+C_2\bigg\|\int_{\gamma}&\bigg(R(\lambda,P^{u_t,v_s}\slashed{D}^{\pi\circ u_t}(P^{u_t,v_s})^{-1})\circ\left(P^{u_t,v_s}\slashed{D}^{\pi\circ u_t}(P^{u_t,v_s})^{-1}-\slashed{D}^{\pi\circ v_s}\right)\circ\\
&\quad R(\lambda,\slashed{D}^{\pi\circ v_s})\bigg)P^{u_0,v_s}\psi_0\bigg\|_{W^{1,p}(M)}
\end{split}\\
&\leq C_2\int_{\gamma}\bigg\|R(\lambda,P^{u_t,v_s}\slashed{D}^{\pi\circ u_t}(P^{u_t,v_s})^{-1})\bigg(P^{u_t,v_s}P^{u_0,u_t}\psi_0-P^{u_0,v_s}\psi_0\bigg)\bigg\|_{W^{1,p}(M)}\\
&\begin{split}
+C_2\int_{\gamma}&\bigg\|\bigg(R(\lambda,P^{u_t,v_s}\slashed{D}^{\pi\circ u_t}(P^{u_t,v_s})^{-1})\circ\left(P^{u_t,v_s}\slashed{D}^{\pi\circ u_t}(P^{u_t,v_s})^{-1}-\slashed{D}^{\pi\circ v_s}\right)\circ\\
&\quad R(\lambda,\slashed{D}^{\pi\circ v_s})\bigg)P^{u_0,v_s}\psi_0\bigg\|_{W^{1,p}(M)}
\end{split}\\
&\leq C_3\sup\limits_{{\rm Im}(\gamma)}\|R(\lambda,P^{u_t,v_s}\slashed{D}^{\pi\circ u_t}(P^{u_t,v_s})^{-1})\|_{L(L^p,W^{1,p})}\|P^{u_t,v_s}P^{u_0,u_t}\psi_0-P^{u_0,v_s}\psi_0\|_{L^p}\\
&\quad+C_3\sup\limits_{{\rm Im}(\gamma)}\|R(\lambda,P^{u_t,v_s}\slashed{D}^{\pi\circ u_t}(P^{u_t,v_s})^{-1})\|_{L(L^p,W^{1,p})}\sup\limits_{{\rm Im}(\gamma)}\|R(\lambda,\slashed{D}^{\pi\circ v_s})\|_{L(L^p,W^{1,p})}\\
&\quad\quad\|P^{u_t,v_s}\slashed{D}^{\pi\circ u_t}(P^{u_t,v_s})^{-1}-\slashed{D}^{\pi\circ v_s}\|_{L(W^{1,p},L^p)}\|P^{u_0,v_s}\psi_0\|_{L^p}.
\end{split}
\end{equation*}

Now, we estimate all the terms in the right-hand side of the inequality above. First, by Lemma \ref{resolvents} and Lemma \ref{W1p norm}, we know that all the resolvents above are uniformly bounded. Next, by Lemma \ref{Dirac along maps u}, we have
\begin{equation*}
\|P^{u_t,v_s}\slashed{D}^{\pi\circ u_t}(P^{u_t,v_s})^{-1}-\slashed{D}^{\pi\circ v_s}\|_{L(W^{1,p},L^p)}\leq C(R)\|u_t-v_s\|_{C^0(M,\mathbb{R}^q)}.
\end{equation*}
Finally, by Lemma \ref{PT}, we obtain
\begin{equation*}
\|P^{u_t,v_s}P^{u_0,u_t}\psi_0-P^{u_0,v_s}\psi_0\|_{L^p}\leq C(\epsilon,\psi_0)\|u_t-v_s\|_{C^0(M,\mathbb{R}^q)}.
\end{equation*}
Putting these together, we get \eqref{difference of psi bar after pt}.

Next, we want to show the following estimate which is very close to \eqref{psi lip}.
\begin{equation}\label{psi bar lip}
\|\bar{\psi}^A(u_t)(x)-\bar{\psi}^A(v_s)(x)\|\leq C(R,\epsilon,\psi_0)\|u_t-v_s\|_{C^0(M,\mathbb{R}^q)}.
\end{equation}

In fact, we have
\begin{equation*}
\begin{split}
&\|\bar{\psi}^A(u_t)(x)-\bar{\psi}^A(v_s)(x)\|\\
&\leq\|\bar{\psi}(u_t)(x)-\bar{\psi}(v_s)(x)\|_{\Sigma_xM\otimes\mathbb{R}^q}\\
&\leq\|P^{u_t,v_s}\bar{\psi}(u_t)(x)-\bar{\psi}(v_s)(x)\|_{\Sigma_xM\otimes\mathbb{R}^q}+\|P^{u_t,v_s}\bar{\psi}(u_t)(x)-\bar{\psi}(u_t)(x)\|_{\Sigma_xM\otimes\mathbb{R}^q}\\
&=\|P^{u_t,v_s}\bar{\psi}(u_t)(x)-\bar{\psi}(v_s)(x)\|_{\Sigma_xM\otimes T_{(\pi\circ v_s(x))}N}\\
&\quad+\|P^{u_t,v_s}\bar{\psi}(u_t)(x)-\bar{\psi}(u_t)(x)\|_{\Sigma_xM\otimes\mathbb{R}^q}\\
&\leq C(R,\epsilon,\psi_0)\|u_t-v_s\|_{C^0(M,\mathbb{R}^q)}+\|P^{u_t,v_s}\bar{\psi}(u_t)(x)-\bar{\psi}(u_t)(x)\|_{\Sigma_xM\otimes\mathbb{R}^q}.
\end{split}\end{equation*}
It remains to estimate the last term in the inequality above. To that end, let $\gamma(r):=\exp_{(\pi\circ u_t)(x)}(r\exp^{-1}_{(\pi\circ u_t)(x)}(\pi\circ v_s(x)))$, $r\in[0,1]$, be the unique shortest geodesic of $N$ from $(\pi\circ u_t)(x)$ to $(\pi\circ v_s)(x)$. Let $X\in T_{\gamma(0)}N$ be given and denote by $X(r)$ the unique parallel vector field along $\gamma$ with $X(0)=X$. Then we have
\begin{equation*}
P^{u_t,v_s}X-X=X(1)-X(0)=\int_0^1\frac{dX}{dr}\bigg|_{r=\xi}d\xi=\int_0^1II(\gamma'(r),X(r))dr.
\end{equation*}
Therefore,
\begin{equation*}
\|P^{u_t,v_s}X-X\|_{\mathbb{R}^q}\leq C_1\sup\limits_{r\in[0,1]}\|\gamma'(r)\|_{N}\sup\limits_{r\in[0,1]}\|X(r)\|_{N}=C_1\|\gamma'(0)\|_{N}\|X\|_{N}
\end{equation*}
where $II$ is the second fundamental form of $N$ in $\mathbb{R}^q$ and $C_1$ only depends on $N$. Using \eqref{control distance on N} and the Lipschitz continuity of $\pi$ we get
\begin{equation*}
\|\gamma'(0)\|_{N}\leq d^N((\pi\circ u_t)(x),(\pi\circ v_s)(x))\leq C_2\|u_t(x)-v_s(x)|\|_{\mathbb{R}^q}
\end{equation*}
and
\begin{equation*}
\|P^{u_t,v_s}X-X\|_{\mathbb{R}^q}\leq C_3\|u_t(x)-v_s(x)|\|_{\mathbb{R}^q}\|X\|_{N}.
\end{equation*}
This implies
\begin{equation*}
\|P^{u_t,v_s}\bar{\psi}(u_t)(x)-\bar{\psi}(u_t)(x)\|_{\Sigma_xM\otimes\mathbb{R}^q}\leq C(R,\epsilon,\psi_0)\|u_t(x)-v_s(x)|\|_{\mathbb{R}^q}.
\end{equation*}
Hence, \eqref{psi bar lip} holds.

Now, using \eqref{difference of psi bar after pt} and \eqref{psi bar lip}, we get
\begin{equation*}
\begin{split}
&\|\psi^A(u_t)(x)-\psi^A(v_s)(x)\|\\
&=\bigg\|\frac{\bar\psi^A(u_t)(x)}{\|\bar\psi(u_t)\|_{L^2}}-\frac{\bar\psi^A(u_t)(x)}{\|\bar\psi(v_s)\|_{L^2}}+\frac{\bar\psi^A(u_t)(x)}{\|\bar\psi(v_s)\|_{L^2}}-\frac{\bar\psi^A(v_s)(x)}{\|\bar\psi(v_s)\|_{L^2}}\bigg\|\\
&\leq\frac{\bar\psi^A(u_t)(x)}{\|\bar\psi(u_t)\|_{L^2}\|\bar\psi(v_s)\|_{L^2}}\bigg|\|\bar\psi(v_s)\|_{L^2}-\|\bar\psi(u_t)\|_{L^2}\bigg|\\
&\quad+\frac{1}{\|\bar\psi(v_s)\|_{L^2}}\|\bar\psi^A(u_t)(x)-\bar\psi^A(v_s)(x)\|\\
&=\frac{\bar\psi^A(u_t)(x)}{\|\bar\psi(u_t)\|_{L^2}\|\bar\psi(v_s)\|_{L^2}}\bigg|\|\bar\psi(v_s)\|_{L^2}-\|P^{u_t,v_s}\bar\psi(u_t)\|_{L^2}\bigg|\\
&\quad+\frac{1}{\|\bar\psi(v_s)\|_{L^2}}\|\bar\psi^A(u_t)(x)-\bar\psi^A(v_s)(x)\|\\
&\leq\frac{\bar\psi^A(u_t)(x)}{\|\bar\psi(u_t)\|_{L^2}\|\bar\psi(v_s)\|_{L^2}}\|P^{u_t,v_s}\bar\psi(u_t)-\bar\psi(v_s)\|_{L^2}\\
&\quad+\frac{1}{\|\bar\psi(v_s)\|_{L^2}}\|\bar\psi^A(u_t)(x)-\bar\psi^A(v_s)(x)\|\\
&\leq\bigg(\frac{\bar\psi^A(u_t)(x)}{\|\bar\psi(u_t)\|_{L^2}\|\bar\psi(v_s)\|_{L^2}}+\frac{1}{\|\bar\psi(v_s)\|_{L^2}}\bigg)C(R,\epsilon,\psi_0)\|u_t-v_s\|_{C^0(M,\mathbb{R}^q)}.
\end{split}
\end{equation*}
Then the inequality \eqref{psi lip} follows from Lemma \ref{projection to kernel} and \eqref{psi bar lip}. This completes the proof.

\end{proof}

\begin{comment}

\section{Conflict of interest}

 On behalf of all authors, the corresponding author states that there is no conflict of interest.

  \section{Data Availability Statement}
  Data sharing not applicable to this article as no datasets were generated or analysed during the current study.

\end{comment}

% ----------------------------------------------------------------

% ----------------------------------------------------------------

\bibliographystyle{amsplain}
\bibliography{reference3}

\end{document}